\documentclass[smallextended, nospthms]{svjour3}       
\smartqed  

\usepackage{graphicx}
\usepackage{bm}
\usepackage{amsmath}
\usepackage{color}
\usepackage{amsthm}
\allowdisplaybreaks[4]
 \usepackage{float}
\usepackage{subfloat}
\usepackage{subfig}
\usepackage{caption}
\usepackage{epstopdf}
\usepackage{diagbox}
\newcommand{\vertiii}[1]{{\left\vert\kern-0.25ex\left\vert\kern-
0.25ex\left\vert #1
		\right\vert\kern-0.25ex\right\vert\kern-
0.25ex\right\vert}}
\newtheorem{thm}{Theorem}[section]
\newtheorem{lem}[thm]{Lemma}
\newtheorem{corol}[thm]{Corollary}
\newtheorem{rem}[thm]{Remark}

\begin{document}

\title{A Stabilized Hybrid Mixed Finite Element Method for Poroelasticity
	\thanks{
	The work of C.~Niu and H.~Rui is supported by the National Natural Science Foundation of China Grant No. 11671233. The work of X.~Hu is partially supported by the National Science Foundation under grant DMS-1620063.
	}
}


\author{Chunyan Niu    \and
        Hongxing Rui    \and
        Xiaozhe Hu
}

\authorrunning{C. Niu, H. Rui, and X. Hu} 

\institute{C. Niu and H. Rui  \at
              School of Mathematics, Shandong University, Jinan, 250100, China. \\
               \email{yanzi198929@163.com,\ hxrui@sdu.edu.cn}
           \and
          X. Hu \at
              Department of Mathematics, Tufts University, 503 Boston Ave, Medford, MA 02155, USA.\\
               \email{Xiaozhe.Hu@tufts.edu}
}

\journalname{Computational Geosciences}

\date{Received: date / Accepted: date}

\maketitle

\begin{abstract}
In this work, we consider a hybrid mixed finite element method for Biot's model. The hybrid P1-RT0-P0 discretization of the displacement-pressure-Darcy's velocity system of Biot's model presented in \cite{C. Niu} is not uniformly stable with respect to
the physical parameters, resulting in some issues in numerical simulations. To alleviate such problems, following \cite{V. Girault},
we stabilize the hybrid scheme with face bubble functions and show the well-posedness with respect to physical and discretization parameters,
which provide optimal error estimates of the stabilized method. We introduce a perturbation of the bilinear form of the displacement which allows for the elimination of the bubble functions. Together with eliminating Darcy's velocity by hybridization, we obtain an eliminated system whose
size is the same as the classical P1-RT0-P0 discretization. Based on the well-posedness of the eliminated system, we design
block preconditioners that are parameter-robust. Numerical experiments are presented to confirm the theoretical results of the stabilized scheme as well as the block preconditioners.
\end{abstract}

\section{Introduction}\label{intro}
Poroelasticity theory simultaneously describes the interaction between the deformation and fluid flow in a fluid-saturated porous medium.
Nowadays, poroelasticity theory is well developed, dating back to the pioneering work of Terzaghi~\cite{K. Terzaghi}, who studied a one-dimensional consolidation problem. A general three-dimensional mathematical model was established by Biot~\cite{M.A. Biot41}, who later extended it to anisotropic and nonlinear materials~\cite{M.A. Biot55,M.A. Biot3}. Results on the existence and uniqueness of the solution for these models have been investigated by Showalter in~\cite{R. Showalter} and \v{Z}en\'{\i}\v{s}ek in~\cite{A. Zenisek}. The well-posedness for nonlinear poroelastic models is considered in~\cite{R.Z. Dautov}. Lewis and Schrefler~\cite{R.W. Lewis1} presented a complete work on theoretical, practical and numerical aspects of geomechanical problems. The analysis and numerical simulation of Biot's model has became increasingly popular due to the wide range of applications such biomechanics, petroleum engineering, and environmental engineering.

Numerically solving the poroelasticity model is a challenging task due to its complex coupled nature. There is extensive literature on numerical methods for poroelasticity. For the two-field displacement-pressure formulation, Lewis and Schrefler~\cite{R.W. Lewis2} analyzed one and two-dimensional problems in consolidation and the problem of subsidence in Venice by using a continuous Galerkin (CG) method. Furthermore, Liu~\cite{R. Liu} implemented a discontinuous Galerkin (DG) method for both displacement and pressure.  A weak Galerkin (WG) method for the two-field formulations was considered in~\cite{X. Hu1,Y. Chen}. Since Phillips and Wheeler~\cite{P.J. Phillips1,P.J. Phillips2,P.J. Phillips3} developed a method that couples a CG method for the displacement with a mixed finite element method for the pressure and velocity, more and more researchers have begun to focus on the three-field formulation, which has displacement, Darcy's velocity, and pressure as unknowns. Yi~\cite{S.Y. Yi1} developed a nonconforming finite element method, which has been used to overcome nonphysical oscillations in the pressure variable. Hu et.al.~\cite{X. Hu2} also presented a nonconforming method based on the Crouzeix-Raviart elements for the displacements, lowest order Raviart-Thomas-N\'{e}d\'{e}lec elements for Darcy's velocity~\cite{P.A. Raviart,J.C. Nedelec}, and piecewise constant elements for the pressure. A DG scheme was proposed in~\cite{HongKraus2018} for the three-field formulation.  A WG scheme which couples a WG method for the displacement with a standard mixed finite element method for the pressure and Darcy's velocity, was proposed in~\cite{M. Sun} to avoid the locking. Other numerical schemes, such as least squares mixed finite element methods, were proposed in~\cite{Korsawe,Tchonkova} for the four-field formulation (displacement, stress, Darcy's velocity and pressure). And a new mixed finite element method for both the flow subproblem and the mechanical subproblem was introduced by Yi~\cite{S.Y. Yi14}, and then~\cite{S.Y. Yi17} presented the iteratively coupled solution strategies for the four-field formulation.  More recently, an element-based finite volume formulation was proposed in~\cite{Herminio T.Honorio18} to avoid pressure instabilities in poromechanics, and~\cite{I. Sokolova19} presented a multiscale finite volume method.

For the three-field formulation, one standard approximation scheme, piecewise constant approximation for the pressure and piecewise linear approximation for the displacements and fluid flux may cause instability and non-physical oscillations, a local pressure jump stabilization term was introduced in~\cite{L. Berger15} to ensure stability. One of the most frequently considered low-order schemes for the three-field formulation is P1-RT0-P0, i.e., piecewise linear elements for the displacement, Raviart-Thomas-N\'{e}d\'{e}lec elements for the Darcy's velocity, and piecewise constant elements for the pressure. However, the triple P1-RT0-P0 does not satisfy the Biot-Stokes stability condition uniformly with respect to the discretization and physical parameters of the problem\textcolor{red}{~\cite{V. Girault,C. Rodrigo,HongKraus2018}}. For example, when the permeability is small with respect to the mesh size, volumetric locking may occur~\cite{C. Rodrigo}. Due to the same reason, the hybridized P1-RT0-P0 scheme presented in~\cite{C. Niu} has the same stability issue. Recently, a stabilization strategy based on the macro-element theory and the local pressure jump approach for the hybridized P1-RT0-P0 formulation has been proposed in~\cite{M. Frigo20}, and an efficient solver was introduced to improve computational efficiency. 

In this paper, we apply the stabilization technique introduced in~\cite{C. Rodrigo} and enriching the piecewise linear finite element space with edge/face bubble functions for displacement to stabilize the hybrid P1-RT0-P0 scheme developed in~\cite{C. Niu}. As shown in~\cite{F. Brezzi,D.N. Arnold,H. Egger}, the hybridization technique, which was proposed in~\cite{T.H.H. Pian}, is characterized by the removal of the continuity of the normal component of Darcy's velocity along each element interface. Therefore, one particular advantage is that the degrees of freedom for Darcy's velocity can be eliminated by static condensation, which reduces the computational cost. For our proposed stabilized hybrid scheme, a perturbation of the bilinear form allows for the elimination of the bubble functions. Therefore, both the unknowns of Darcy's velocity and the bubble functions can be eliminated in our case, which results in a system with the same number of degrees of freedom as the standard P1-RT0-P0 discretization. The eliminated system is proved to be well-posed with respect to the discretization as well as the physical parameters, and thus, based on the framework developed in~\cite{D. Loghin,K.A. Mardal,James}, we develop robust block preconditioners to solve the resulting linear systems of equations efficiently. Moreover, due to the introduction of hybridization and eliminations of the fluid flux, in the implementation of the block preconditioners, we avoid solving complicated $\operatorname{grad} \operatorname{div}$-type subproblem, which requires special solvers~\cite{James}, and only need to solve standard $\operatorname{div} \operatorname{grad}$-type subproblem, which can be efficiently handled by standard multigrid methods.

The rest of the paper is organized as follows. In Section~\ref{sec:stable}, we briefly recall the Biot's model, and mixed variational formulation of the three-field formulation, and then devote to a stabilization technique with edge bubble functions and the perturbation of the bilinear form. The well-posedness of the resulting scheme, as well as the corresponding error analysis are also provided. In Section~\ref{sec:elimination} we discuss the elimination of bubbles and Darcy's velocity, and show the well-posedness of the eliminated system. Both the block diagonal and triangular preconditioners are presented in Section~\ref{sec:preconditioner}. Numerical results are presented in Section~\ref{sec:numerical} to validate the accuracy and efficiency of the stabilization method and demonstrate the robustness and effectiveness of the preconditioners. Finally, our conclusion are presented in Section 7.

\section{Stabilized Hybrid Mixed Finite Element Method}\label{sec:stable}
In this section, we first recall the Biot's consolidation model~\cite{M.A. Biot41} in a domain $\Omega \subset R^{d}$ $ (d=2,3)$, and then analyze the stabilized hybrid mixed finite element method.

\subsection{The Biot's Model}

In this subsection, we review Biot's consolidation model and present a three-field formulation of which the primary variables are displacement $\bm{u}$, pressure $p$ and Darcy's velocity $\bm{w}$.
\begin{equation} \label{e:MODEL}
\begin{array}{l}
-\nabla \cdot \bm{\sigma}^{\prime}+\alpha \nabla p =\bm{f},  \hspace{3.5mm} \text{in} \ \Omega, \\
 \bm{\sigma}^{\prime}=2 \mu \epsilon(\bm{u})+\lambda\ (\nabla \cdot\bm{u})\bm{I},  \hspace{3.5mm} \text{in} \ \Omega, \\
\displaystyle{\frac{\partial}{\partial t}(\frac{1}{M}p+\alpha\nabla \cdot \bm{u})+\nabla \cdot \bm{w}=g}, \hspace{3.5mm}\text{in} \ \Omega, \\
\bm{w}=-\kappa\nabla p, \hspace{3.5mm}\text{in} \ \Omega, \\
\end{array}
\end{equation}
where $\kappa$ stands for the permeability tensor, $M$ is the Biot modulus, and $\alpha$ is the Biot-Willis constant. $\lambda$ and $\mu$ are Lam\'{e} constants. Here, the effective stress tensor and the strain tensor are denoted by $\bm{\sigma}^{\prime}$ and $\epsilon(\bm{u})=\frac{1}{2}(\nabla \bm{u}+\nabla \bm{u}^{\top})$, and $\bm{I}$ is the identity tensor. Let $g$ be the volumetric source/sink term, and $\bm{f}$ represents the body force. The boundary conditions are
\[  p=0,  \hspace{3.5mm} \textup{ on} \ \overline{\Gamma}_{t},\ \ \ \ \ -\kappa\nabla p\cdot\textbf{n}=0,\hspace{3.5mm} \textup{on} \ \Gamma_{c},\]
\[ \bm{u}=0, \hspace{3.5mm} \textup{on} \ \overline{\Gamma}_{c}, \ \ \ \ \ \bm{\sigma}^{\prime} \textbf{n}=0,\hspace{3.5mm} \textup{on} \ \Gamma_{t},\]
 where \textbf{n} represents the outward unit normal vector to the boundary, $\overline{\Gamma}_{t}\cup \overline{\Gamma}_{c}=\overline{\partial \Omega}$
  with $\Gamma_{c}$ and $\Gamma_{t}$ being open (with respect to $\partial \Omega$) subsets of $\partial \Omega$ with nonzero measure.
   The initial condition at $t=0$ is given by,
   \[
p(\bm{x},0)=p_{0},\hspace{3.5mm} \bm{x} \in \Omega,\
\ \ \ \ \ \bm{u}(\bm{x},0)=\bm{u}_{0},\hspace{3.5mm}\bm{x} \in \Omega.
\]
Next, we give the function spaces which are used in the variational form:
\begin{eqnarray*}
&&\bm{V}=\{\bm{v} \in \bm{H}^{1}(\Omega): \bm{v}|_{\overline{\Gamma}_{c}}=\bm{0}\},\\
&&Q=L^{2}(\Omega),\\
&&\bm{W}=\{\bm{w} \in \bm{H}(\textup{div},\Omega): \bm{w}\cdot\textbf{n}|_{\Gamma_{c}}=0\},
\end{eqnarray*}
where $\bm{H}^{1}(\Omega)$ is the space of square integrable vector-valued functions whose first derivatives are also square
integrable, and $\bm{H}(\textup{div}, \Omega)$ contains the square integrable vector-valued functions with square integrable divergence.

Finally, we give the mixed variational formulation: For all $ t \in (0,t_{max}]$, find $( \bm{u},p,\bm{w})\in \bm{V}\times Q\times\bm{W} $, such that,
\begin{eqnarray*}
&&a(\bm{u},\bm{v})-\alpha(p,\nabla \cdot \bm{v})=(\bm{f},\bm{v}), \ \ \forall \ \bm{v} \in \bm{V},\\
&&\displaystyle{\alpha(\nabla \cdot\frac{ \partial\bm{u}}{\partial t},q)+\frac{1}{M}(\frac{\partial p}{\partial t},q)+(\nabla \cdot \bm{w},q)=(g,q)}, \ \  \forall \ q \in Q,\\
&&(\kappa^{-1}\bm{w},\bm{r})-(p,\nabla \cdot \bm{r})=0,  \ \  \forall \ \bm{r} \in \bm{W},
\end{eqnarray*}
where $a(\bm{u},\bm{v})=2\mu \int_{\Omega}\epsilon(\bm{u}):\epsilon(\bm{v})+\lambda \int_{\Omega}\nabla\cdot\bm{u}\nabla\cdot\bm{v}.$
Well-posedness of the continuous problem for this three field formulation was established by Lipnikov~\cite{K. Lipnikov}.

\subsection{Stabilized Hybrid Mixed Finite Element Method}\label{subsec:stable}
Let $ \{\mathcal{T}_{h}\}$ be a quasi-uniform regular partition of $\Omega$, the element $T \in {\mathcal{T}_{h}}$ is triangular (d=2)
or tetrahedral (d=3). Denote $\partial \mathcal{T}_{I}$ the set of interior edges, i.e., the set of common edges of $T^{+}$ $\cap$ $ T^{-} $ for all
neighboring $T^{+}, T^{-}\in {\mathcal{T}_{h}}$. Correspondingly, let $\partial \mathcal{T}_{B}$ denote the set of boundary edges, i.e., the set of the common edges of $T \cap \partial\Omega$ for all $ T \in{\mathcal{T}_{h}}$, and denote $ \partial {\mathcal{T}_{h}}= \partial \mathcal{T}_{I}\cup \partial\mathcal{T}_{B}$. For every edge $e \in \partial \mathcal{T}_{h}$, we associate it with a unit outer normal $\bm{n}_{e}$. For the boundary faces $e \in \partial \mathcal{T}_{B}$, we set $\bm{n}_{e}=\pm\bm{n}_{e,T}$, where $\bm{n}_{e,T}$ is the outward (with respect to $T \in \mathcal{T}_{h}$) unit normal vector to the edge $e$. For the interior edges $ e \in \partial \mathcal{T}_{I}$, the direction of $\bm{n}_{e}$ is fixed, while the particular direction of $\bm{n}_{e}$ is not important. Hence, we define $\bm{n}_{e}=\bm{n}_{e,T^{+}}=-\bm{n}_{e,T^{-}}$.

 Now we introduce the piecewise Sobolev spaces:
\begin{eqnarray*}
&&L^{2}(\mathcal{T}_{h})=\{p \in L^{2}(\Omega):p|_{T}\in L^{2}(T),\forall \ T \in \mathcal{T}_{h}\},\\
&&H^{1}(\mathcal{T}_{h})=\{p \in L^{2}(\Omega):p|_{T}\in H^{1}(T),\forall \ T \in \mathcal{T}_{h}\},\\
&&L^{2}(\partial \mathcal{T}_{h})=\{p \in  L^{2}(e),\forall \ e \in \partial\mathcal{T}_{h}\},
\end{eqnarray*}
where $L^{2}(T)$ denotes the set of square integrable functions on element $T$, $H^{1}(T)$ is the space of square integrable functions on element $T$
whose first derivatives are also square integrable, and $L^{2}(e)$ the set of square integrable functions on edge $e$.
The norm endowed with $L^{2}(\mathcal{T}_{h})$ is defined as $\|p\|_{h}:=\sqrt{(p,p)_{h}}$, and the corresponding inner product is defined as
\begin{eqnarray*}
 (p,q)_{h}=\sum\limits_{T \in \mathcal{T}_{h}}(p,q)_{T},\ \ (p,q)_{T}=\int_{T}pq dx.
\end{eqnarray*}
For $ \beta,\rho \in L^{2}(\partial \mathcal{T}_{h})$, we define the inner product on element interfaces as follows,
\begin{eqnarray*}
 (\beta,\rho)_{\partial\mathcal{T}_{h}}=\sum\limits_{T \in \mathcal{T}_{h}}(\beta,\rho)_{\partial T},\ \ (\beta,\rho)_{\partial T}=\int_{\partial T}\beta \rho ds,
\end{eqnarray*}
and the corresponding norms are denoted by $|\rho|_{\partial\mathcal{T}_{h}}:=\sqrt{(\rho,\rho)_{\partial\mathcal{T}_{h}}}$. For the vector-valued functions, the norm and inner product are defined in a similar manner with corresponding modifications.

Following the standard idea of hybrid mixed finite element method, we use completely discontinuous piecewise polynomial functions for the Darcy's
velocity and ensure the continuity of the normal fluxes over element interfaces by adding appropriate constraints and introducing a Lagrange multiplier.
We introduce the discrete finite element spaces for the displacement, pressure, Lagrange multiplier and Darcy's velocity,
\begin{equation}\label{spacesQBW}
\begin{array}{l}
\bm{V}_{h,1}=\{ \bm{v}_{h} \in \bm{V}:\bm{v}_{h}|_{T} \in [P_{1}(T)]^{d},\ \bm{v}_{h}|_{\partial \mathcal{T}_{B}} =\bm{0},\ \forall \ T \in \mathcal{T}_{h} \},\\
Q_{h}=\{ q_{h} \in L^{2}(\mathcal{T}_{h}):q_{h}|_{T} \in P_{0}(T),\forall \ T \in \mathcal{T}_{h} \}, \\
B_{h}=\{\rho_{h} \in L^{2}(\partial \mathcal{T}_{h}):\rho_{h} |_{e} \in P_{0}(e),\forall \ e \in \partial \mathcal{T}_{h} \},\\
\bm{W}_{h}=\{ \bm{w}_{h} \in [H^{1}(\mathcal{T}_{h})]^{d}:\bm{w}_{h} |_{T} \in RT_{0}(T),\bm{w}_{h}\cdot \bm{n}_{e}|_{\partial \mathcal{T}_{B}}=0, \ \forall \ T \in \mathcal{T}_{h} \},
\end{array}
\end{equation}
where $P_{0}(T)$ denotes the set of piecewise constant functions restricted to $T$ for each $T \in \mathcal{T}_{h}$ and $RT_{0}(T)=[P_{0}(T)]^{d}\oplus
\textup{span}(\bm{x} P_{0}(T))$ denotes the standard lowest order Raviart-Thomas-N\'{e}d\'{e}lec space.

For uniformly positive definite permeability tensor $\kappa$, the choice of spaces $\bm{V}_{h,1}$ has been successfully employed for numerical simulations of Biot's consolidation model (see~\cite{P.J. Phillips2,K. Lipnikov}). However, the heuristic considerations that expose some of the issues with this discretization are observed when $\kappa \rightarrow \bm{0}$. In such case, the discrete problem approaches a P1-P0 discretization of the Stokes' equation. As it is well known, the finite element pair, $\bm{V}_{h,1} \times Q_{h}$, does not satisfy the inf-sup condition and is unstable for the Stokes' problem.

Following~\cite{V. Girault,C. Rodrigo}, we design a stabilized hybrid mixed finite element method by enriching the piecewise linear finite element space $\bm{V}_{h,1}$, with edge bubble functions in 2D or face bubble functions in 3D. Now we introduce the stabilized finite element space $\bm{V}_{h}$ as
\begin{equation}\label{Vh}
\displaystyle{\bm{V}_{h}=\bm{V}_{h,1}\oplus\bm{V}_{b},\ \ \ \bm{V}_{b}=\text{span}\{\bm{\Phi}_{e}\}_{e \in \partial \mathcal{T}_{0,t}}},
\end{equation}
where $\bm{\Phi}_{e}=\varphi_{e}\bm{n}_{e}$, and $\varphi_{e}|_{T^{\pm}}=\varphi_{e,T^{\pm}}=\displaystyle{\prod_{l=1,l\neq j^{\pm}}^{d+1}
\lambda_{l,T^{\pm}}}$, for every face $e \in \partial \mathcal{T}_{h}$, $e = T^{+}\cap T^{-}$. Here, $\lambda_{l,T^{\pm}},l=1,...,(d+1)$ are
barycentric coordinates on $T^{\pm}$ and $j^{\pm}$ is the vertex opposite to face $e$ in $T^{\pm}$. Note that $\bm{\Phi}_{e}\in \bm{V}$ is a continuous piecewise polynomial function of degree $d$.

The degree of freedom associated with $\bm{V}_{h}$ are the values at the vertices of $\mathcal{T}_{h}$ and the total flux through $e \in \partial
\mathcal{T}_{0,t}$. The canonical interpolant $\Pi:C(\overline{\Omega})\mapsto \bm{V}_{h}$ is defined as:
\begin{equation*}
\begin{array}{l}
\displaystyle{\Pi\bm{v}=\Pi_{1}\bm{v}+\sum_{e\in\partial \mathcal{T}_{0,t}}\nu_{e}\bm{\Phi}_{e}},
\end{array}
\end{equation*}
where $\Pi_{1}:C(\overline{\Omega})\mapsto \bm{V}_{h,1}$ is the standard piecewise linear interpolant, and $\displaystyle{\nu_{e}=
\frac{1}{|e|}\int_{e}(I-\Pi_{1})\bm{v}}$.

 Finally, by using backward Euler time discretization with constant time step size $\tau$, the stabilized hybrid mixed finite element method is, for all $t_{n}=n \tau, n=1,2,...,$ find
 $(\bm{u}_{h}^{n},p_{h}^{n},\beta_{h}^{n},\bm{w}_{h}^{n})\in \bm{V}_{h}\times Q_{h}\times B_{h}\times\bm{W}_{h} $, such that,
\begin{equation}\label{e:stableHMFE}
\begin{array}{l}
a(\bm{u}_{h}^{n},\bm{v}_{h})-\alpha(p_{h}^{n},\nabla \cdot \bm{v}_{h})
=(\bm{f},\bm{v}_{h}),\  \forall \bm{v}_{h} \in \bm{V}_{h},\\
\displaystyle{\alpha(\nabla \cdot\bm{u}_{h}^{n},q_{h})+\frac{1}{M}( p_{h}^{n},q_{h}) +\tau(\nabla \cdot \bm{w}_{h}^{n},q_{h}) =\tau (\widetilde{g},q_{h})}, \ \forall  q_{h} \in Q_{h},\\
 \tau(\bm{w}_{h}^{n}\cdot \bm{n}_{e},\rho_{h})_{\partial\mathcal{T}_{h}}=0,  \ \forall \rho_{h} \in B_{h},\\
-\tau(p_{h}^{n},\nabla \cdot \bm{r}_{h})-\tau( \beta_{h}^{n},\bm{r}_{h}\cdot \bm{n}_{e})_{\partial\mathcal{T}_{h}}+\tau(\kappa^{-1}\bm{w}_{h}^{n},\bm{r}_{h})_{h}=0, \ \forall \bm{r}_{h} \in \bm{W}_{h},
\end{array}
\end{equation}
where $(\widetilde{g},q_{h})=\tau (g,q_{h})+\frac{1}{M}( p_{h}^{n-1},q_{h}) +\alpha(\nabla \cdot\bm{u}_{h}^{n-1},q_{h})$, $(\bm{u}_{h}^{n},p_{h}^{n},\beta_{h}^{n},\bm{w}_{h}^{n})\approx(\bm{u}(\cdot,t_{n}),p(\cdot,t_{n}),\beta(\cdot,t_{n}),\bm{w}(\cdot,t_{n}))$, $t_{n}=n\tau$, $n=1,2,...$, and the initial conditions are $\bm{u}_{h}^{1}=\bm{u}_{0}$, $p_{h}^{1}=p_{0}$.

\begin{rem}
When the space $\bm{V}_{h,1}$ is used for displacement, \eqref{e:stableHMFE} is just the hybrid mixed finite element method presented in~\cite{C. Niu}:
 For all $t_{n}=n \tau, n=1,2,...,$ find $(\bm{u}_{h}^{n},p_{h}^{n},\beta_{h}^{n},\bm{w}_{h}^{n})\in \bm{V}_{h,1}\times Q_{h}\times B_{h}\times
\bm{W}_{h} $, such that,
\begin{equation}\label{e:HMFE}
\begin{array}{l}
a(\bm{u}_{h}^{n},\bm{v}_{h})-\alpha(p_{h}^{n},\nabla \cdot \bm{v}_{h})
=(\bm{f},\bm{v}_{h}),\ \forall  \bm{v}_{h} \in \bm{V}_{h,1},\\
\displaystyle{\alpha(\nabla \cdot\bm{u}_{h}^{n},q_{h})+\frac{1}{M}( p_{h}^{n},q_{h}) +\tau(\nabla \cdot \bm{w}_{h}^{n},q_{h}) =(\widetilde{g},q_{h})}, \ \forall  q_{h} \in Q_{h},\\
 \tau(\bm{w}_{h}^{n}\cdot \bm{n}_{e},\rho_{h})_{\partial\mathcal{T}_{h}}=0, \ \forall  \rho_{h} \in B_{h},\\
-\tau(p_{h}^{n},\nabla \cdot \bm{r}_{h})-\tau( \beta_{h}^{n},\bm{r}_{h}\cdot \bm{n}_{e})_{\partial\mathcal{T}_{h}}+\tau(\kappa^{-1}\bm{w}_{h}^{n},\bm{r}_{h})_{h}=0, \ \forall  \bm{r}_{h} \in \bm{W}_{h}.
\end{array}
\end{equation}
\end{rem}

\subsection{Well-posedness}
To show that the discrete problem (\ref{e:stableHMFE}) is well-posed, we introduce a bilinear form on the space $S_{h}:=\bm{V}_{h}\times Q_{h}\times B_{h}\times\bm{W}_{h}$,
\begin{equation*}\label{e:B}
\begin{array}{l}
\mathcal{B}((\bm{u}_{h},p_{h},\beta_{h},\bm{w}_{h}),(\bm{v}_{h},q_{h},\rho_{h},\bm{r}_{h}))\\
:=a(\bm{u}_{h},\bm{v}_{h})-\alpha(p_{h},\nabla \cdot \bm{v}_{h})+\alpha(\nabla \cdot \bm{u}_{h},q_{h})+\displaystyle{\frac{1}{M}(p_{h},q_{h})}+\tau(\nabla \cdot\bm{w}_{h}, q_{h})\\
\hspace{4mm}
+\tau(\bm{w}_{h}\cdot \bm{n}_{e},\rho_{h})_{\partial\mathcal{T}_{h}}-\tau( p_{h},\nabla \cdot\bm{r}_{h})
-\tau(\beta_{h},\bm{r}_{h}\cdot \bm{n}_{e})_{\partial\mathcal{T}_{h}}+\tau(\kappa^{-1}\bm{w}_{h},\bm{r}_{h})_{h},
\end{array}
\end{equation*}
and a norm,
\begin{equation*}
\begin{array}{l}
\vertiii{(\bm{u}_{h},p_{h},\beta_{h},\bm{w}_{h})}
:= \displaystyle{\biggl[\|\bm{u}_{h}\|_{a}^{2}+\delta\|p_{h}\|^{2}+\frac{1  }{h}|\beta_{h}|_{\partial\mathcal{T}_{h}}^{2}
+\tau\|\bm{w}_{h}\|_{h,\kappa^{-1}}^{2}}\\
\hspace{37mm}+\displaystyle{\frac{\tau^{2}}{\delta} \|\nabla\cdot \bm{w}_{h}\|_{h}^{2}+\tau^{2} h\|\bm{w}_{h}\cdot \bm{n}_{e}\|_{h}^{2}\biggr]^{1/2}},\\
\end{array}
\end{equation*}
where $\displaystyle{\zeta=\sqrt{\lambda+2 \mu/d}}$, $\displaystyle{\delta=\frac{\alpha^{2}}{\zeta^{2}}+\frac{1}{M}}$, and $\|\bm{w}\|_{\mathcal{T}_{h},\kappa^{-1}}
=(\kappa^{-1}\bm{w},\bm{w})^{1/2}$. Then, we have the following theorem to show the bilinear form $\mathcal{B}(\cdot,\cdot)$ is inf-sup stable and continuous, which shows that the hybrid mixed finite element method is well-posed at each time step.

\begin{thm} \label{t:stab}
 If the finite element spaces $Q_{h}, B_{h}, \bm{W}_{h}$  are chosen as (\ref{spacesQBW}), and $\bm{V}_{h}$ as (\ref{Vh}), the following \emph{inf-sup} condition holds
 \begin{equation*}\label{wellposedness}
\begin{array}{l}
 \displaystyle{\sup_{(\bm{v}_{h},q_{h},\rho_{h},\bm{r}_{h})\in S_{h}}}
 \frac{\mathcal{B}((\bm{u}_{h},p_{h},\beta_{h},\bm{w}_{h}),(\bm{v}_{h},q_{h},\rho_{h},\bm{r}_{h})) }{\vertiii{(\bm{v}_{h},q_{h},\rho_{h},\bm{r}_{h})}}
\geq  \gamma_{\ast}\vertiii{(\bm{u}_{h},p_{h},\beta_{h},\bm{w}_{h})}, \ \
\end{array}
\end{equation*}
and, $\mathcal{B}(\cdot , \cdot)$ is continuous with respect to $\vertiii{\cdot}$, i.e.,
\begin{equation*}\label{continuous}
\begin{array}{l}
 |\mathcal{B}((\bm{u}_{h},p_{h},\beta_{h},\bm{w}_{h}),(\bm{v}_{h},q_{h},\rho_{h},\bm{r}_{h}))|\leq \gamma^{\ast}
\vertiii{(\bm{u}_{h},p_{h},\beta_{h},\bm{w}_{h})}\vertiii{(\bm{v}_{h},q_{h},\rho_{h},\bm{r}_{h})}. \ \
\end{array}
\end{equation*}
Here, the constants $\gamma_{\ast}$, $ \gamma^{\ast}> 0$ are independent of mesh size $h$, time step size $\tau$, and the physical parameters.
\end{thm}

\begin{proof}
First, for $d=2,3$, $2\mu (\epsilon(\bm{u}):\epsilon(\bm{v}))\leq a(\bm{u},\bm{v})\leq( 2\mu+d \lambda) (\epsilon(\bm{u}):\epsilon(\bm{v}))$, then $(\nabla \cdot \bm{v},\nabla \cdot \bm{v})\leq d (\epsilon(\bm{u}):\epsilon(\bm{v}))$, and
\begin{equation*}
\begin{array}{l}
\displaystyle{( \frac{2\mu}{d}+ \lambda)\|\nabla \cdot \bm{v}\|^{2} \leq \| \bm{v}\|_{a}^{2}
\Rightarrow \|\nabla \cdot \bm{v}\|\leq \frac{1}{\zeta}\| \bm{v}\|_{a}}.\\
\end{array}
\end{equation*}
According to the definition of $\zeta$, we have $\| \bm{v}\|_{a}\leq \sqrt{d}\zeta\| \bm{v}\|_{1}$. Due to our choice of finite element spaces, for a given $p \in Q_{h}$, there exists an $\bm{h} \in \bm{V}_{h}$ such that $\displaystyle{(p, \nabla \cdot \bm{h})\geq \frac{\gamma_{B}}{\zeta}\|p\|^{2}}$ and $\|\bm{h}\|_{a}=\|p\|$~\cite{V. Girault}.

Next, following~\cite{J. Schoberl}, let $\bm{z}_{\bm{n}}\in H^{-1/2}(T)$, we solve the weak form of the scalar equation $\displaystyle{-\frac{1}{\tau}
\nabla\cdot\kappa\nabla u +\frac{\theta^{2}}{\tau^{2}}u=0}$ with boundary conditions $\displaystyle{\kappa\frac{\partial u}{\partial n}}$
  $\displaystyle{=\bm{z}_{\bm{n}}}$, and there exists a unique solution $u \in H^{1}(T)$ such that,
\begin{equation*}
\frac{1}{\tau}\|\nabla u\|^{2}_{\kappa}+\frac{\theta^{2}}{\tau^{2}}\| u\|^{2}\leq C\|\bm{z}_{\bm{n}}\|^{2}_{H^{-1/2}(T)},
\end{equation*}
where, $\|\nabla u\|_{\kappa}=(\kappa \nabla u,\nabla u)^{1/2}$.
Set $\bm{z}=\kappa\nabla u$, then $\nabla \cdot \bm{z}=\kappa \triangle u=\displaystyle{\frac{\theta^{2}}{\tau}u}$ and $\bm{z}_{\bm{n}}=\bm{z} \cdot \bm{n}_{e}$, thus,
\begin{equation*}
\frac{1}{\tau}\|\kappa^{-1}\bm{z}\|^{2}_{\kappa}+\frac{\theta^{2}}{\tau^{2}}\| \frac{\tau}{\theta^{2}}\nabla \cdot \bm{z}\|^{2}\leq C\|\bm{z} \cdot \bm{n}_{e}\|^{2}_{H^{-1/2}(T)}.
\end{equation*}
Immediately, we have,
\begin{equation}\label{zn}
 \frac{1}{\tau^{2}}\biggl(\tau\|\bm{z}\|^{2}_{h,\kappa^{-1}}+\frac{\tau^{2}}{\theta^{2}}\|\nabla \cdot \bm{z}\|^{2}\biggr)\leq C\|\bm{z} \cdot \bm{n}_{e}\|^{2}_{H^{-1/2}(T)}.
\end{equation}
Introduce the interpolation operator $\Pi^{\text{div}}_{T}$~\cite{F. Brezzi}:
  $H(\text{div},T)\rightarrow RT_{0}(T)$ as $\bm{z}_{h}|_{T}=\Pi^{\text{div}}_{T}\bm{z}|_{T}$. Let $\Pi^{L^{2}}_{T}$ be the $L^{2}$ projection on $P_{0}(T)$,
  we have $\nabla \cdot\bm{z}_{h}|_{T}=\nabla \cdot\Pi^{\text{div}}_{T}\bm{z}|_{T}=\Pi^{L^{2}}_{T}\nabla \cdot \bm{z}|_{T}$. Let $\bm{z}_{h} \cdot \bm{n}_{e}=\displaystyle{\frac{1}{\tau h} \beta_{h}\in B_{h}}$, then we have
 \begin{equation}\label{zh}
 \begin{array}{l}
\displaystyle{\tau\|\bm{z}_{h}\|^{2}_{h,\kappa^{-1}}+\frac{\tau^{2}}{\theta^{2}}\| \nabla \cdot \bm{z}_{h}\|_{h}^{2}}\\
\displaystyle{=\sum_{T\in \mathcal{T}_{h}}\biggl(\tau\|\Pi^{\text{div}}_{T}\bm{z}\|^{2}_{T,\kappa^{-1}}+\frac{\tau^{2}}{\theta^{2}} \| \Pi^{L^{2}}_{T}\nabla \cdot \bm{z}\|_{T}^{2}\biggr)}\\
\displaystyle{\leq\sum_{T\in \mathcal{T}_{h}}\biggl(\tau\|\bm{z}\|^{2}_{T,\kappa^{-1}}+\frac{\tau^{2}}{\theta^{2}}\|\nabla \cdot \bm{z}\|_{T}^{2}\biggr)}\\
\displaystyle{\leq\sum_{T\in \mathcal{T}_{h}}\frac{C}{h} \|\beta_{h}\|_{L^{2}(T)}^{2}}\\
\displaystyle{\leq \frac{C}{h} |\beta_{h}|_{\partial\mathcal{T}_{h}}^{2}},
\end{array}
\end{equation}
where the properties of the operators $\Pi^{\text{div}}_{T}$ and $\Pi^{L^{2}}_{T}$~\cite{R. Hiptmair} derives the first inequality, the second inequality comes from (\ref{zn}) directly and the third inequality is derived from the property of space $H^{-1/2}$~\cite{J. Schoberl}.

Let $\bm{v}_{h}=\bm{u}_{h}-\theta\bm{h}$, $q_{h}=p_{h}+\theta_{1}\tau\nabla \cdot\bm{w}_{h}$, $\rho_{h}=\beta_{h}+\theta_{2}\tau h \bm{w}_{h}\cdot
\bm{n}_{e}$, and $\bm{r}_{h}=\bm{w}_{h}-\theta_{3}\bm{z}_{h}$, then by the Cauchy-Schwarz inequality and (\ref{zh}), we have,
\begin{align*}
&\mathcal{B}((\bm{u}_{h},p_{h},\beta_{h},\bm{w}_{h}),(\bm{v}_{h},q_{h},\rho_{h},\bm{r}_{h})) \\
&=\|\bm{u}_{h}\|_{a}^{2}-\theta a(\bm{u}_{h},\bm{h})+\theta\alpha(p_{h},\nabla \cdot \bm{h})
+\theta_{1}\alpha\tau(\nabla \cdot \bm{u}_{h},\nabla \cdot\bm{w}_{h})+\displaystyle{\frac{1}{M}}\|p_{h}\|^{2}\\
&\hspace{4mm}
+\displaystyle{\theta_{1}\frac{\tau}{M}} (p_{h},\nabla \cdot\bm{w}_{h})
+\theta_{1}\tau^{2}\|\nabla \cdot\bm{w}_{h}\|_{h}^{2}
+\theta_{2}\tau^{2} h\|\bm{w}_{h}\cdot \bm{n}_{e}\|_{h}^{2}+\theta_{3}\tau(p_{h},\nabla \cdot\bm{z}_{h})\\
&\hspace{4mm}
+\theta_{3}\tau(\beta_{h},\bm{z}_{h}\cdot \bm{n}_{e})_{\partial\mathcal{T}_{h}}
+\tau\|\bm{w}_{h}\|_{h,\kappa^{-1}}^{2}-\theta_{3}\tau(\kappa^{-1}\bm{w}_{h},\bm{z}_{h})_{h} \\
&\geq \|\bm{u}_{h}\|_{a}^{2}-\displaystyle{\frac{1}{2}\|\bm{u}_{h}\|_{a}^{2}-\frac{\theta^{2}}{2}\|\bm{h}\|_{a}^{2}+
\theta\frac{\alpha\gamma_{B}}{\zeta}}\|p_{h}\|^{2}
-\displaystyle{\frac{\theta_{1}\alpha^{2}}{2}\|\nabla \cdot\bm{u}_{h}\|^{2}}\\
&\hspace{4mm}
\displaystyle{-\frac{\theta_{1}\tau^{2}}{2}\|\nabla \cdot\bm{w}_{h}\|_{h}^{2}
+\frac{1}{M}\|p_{h}\|^{2}
-\frac{3\theta_{1}}{2{M}^{2}}\|p_{h}\|^{2}-\frac{\theta_{1}\tau^{2}}{6}\|\nabla \cdot\bm{w}_{h}\|_{h}^{2}}\\
&\hspace{4mm}
\displaystyle{+\theta_{1}\tau^{2}\|\nabla \cdot\bm{w}_{h}\|_{h}^{2}}
+\displaystyle{\theta_{2}\tau^{2} h\|\bm{w}_{h}\cdot \bm{n}_{e}\|_{h}^{2}
-\frac{\theta^{2}}{2}\|p_{h}\|^{2}-\frac{\theta_{3}^{2}\tau^{2}}{2\theta^{2}}\|\nabla \cdot\bm{z}_{h}\|_{h}^{2}}\\
&\hspace{4mm}
\displaystyle{+\frac{\theta_{3}}{h} |\beta_{h}|_{\partial\mathcal{T}_{h}}^{2}
+\tau\|\bm{w}_{h}\|_{h,\kappa^{-1}}^{2}-\frac{\tau}{2}\|\bm{w}_{h}\|_{h,\kappa^{-1}}^{2}-\frac{\theta_{3}^{2}\tau}{2}\|\bm{z}_{h}\|_{h,\kappa^{-1}}^{2}}\\
& \geq \displaystyle{\biggl(\frac{1}{2}-\frac{\theta_{1}\alpha^{2}}{2\zeta^{2}}\biggr)\|\bm{u}_{h}\|_{a}^{2}+\biggl(\frac{1}{M}-\frac{3\theta_{1}}{2{M}^{2}}
 +\frac{\theta \alpha\gamma_{B}}{\zeta}-\frac{\theta^{2}}{2}-\frac{\theta^{2}}{2}\biggr)\|p_{h}\|^{2}}\\
&\hspace{4mm}
 \displaystyle{+\biggl(\theta_{3}
 -\frac{C\theta_{3}^{2}}{2}\biggr)\frac{1}{h}  |\beta_{h}|_{\partial\mathcal{T}_{h}}^{2}+\frac{\tau}{2}\|\bm{w}_{h}\|_{h,\kappa^{-1}}^{2}+\frac{\tau^{2}\theta_{1}}{3} \|\nabla \cdot \bm{w}_{h}\|_{h}^{2}
 +\theta_{2}\tau^{2} h\|\bm{w}_{h}\cdot \bm{n}_{e}\|_{h}^{2}}\\
&\geq \gamma_{1}\vertiii{(\bm{w}_{h},p_{h},\beta_{h},\bm{u}_{h})}^{2},
\end{align*}
where $\displaystyle{\gamma_{1}=\min\{ \frac{ \gamma_{B}^{2}}{4}, \frac{1}{2C},\frac{1}{6}\}}$, and we choose $\displaystyle{\theta=\frac{ \alpha\gamma_{B}}{2\zeta}}$, $\displaystyle{\theta_{1}=\frac{1}{2\delta}}$, $\theta_{2}=1$, and $\displaystyle{\theta_{3}=\frac{1}{C}}$.
Then by the Cauchy-Schwarz inequality, we have,
\begin{align*}
&\vertiii{(\bm{v}_{h},q_{h},\rho_{h},\bm{r}_{h})}^2\\
&=\|\bm{u}_{h}-\theta\bm{h}\|_{a}^{2}+\delta\|p_{h}+\theta_{1}\tau\nabla \cdot\bm{w}_{h}\|_{h}^{2}
\displaystyle{+\frac{1}{h}|\beta_{h}+\theta_{2}\tau h\bm{w}_{h}\cdot\bm{n}_{e}|_{\partial\mathcal{T}_{h}}^{2}}\\
&\hspace{4mm}
+\displaystyle{\tau\|\bm{w}_{h}-\theta_{3}\bm{z}_{h}\|^{2}_{h,\kappa^{-1}}+\frac{\tau^{2}}{\delta} \|\nabla \cdot (\bm{w}_{h}-\theta_{3}\bm{z}_{h})\|_{h}^{2}+\tau^{2} h \|(\bm{w}_{h}-\theta_{3}\bm{z}_{h})\cdot \bm{n}_{e}\|_{h}^{2}}\\
&\leq\displaystyle{2\|\bm{u}_{h}\|_{a}^{2}+2\biggl(\delta+\frac{\gamma_{B}^{2} \alpha^{2}}{4\zeta^{2}}\biggr)\|p_{h}\|^{2}
 +\biggl(\frac{2}{C}+\frac{\gamma_{B}^{2}}{2C}+\frac{2}{C^{2}}+2\biggr)\frac{1}{h}|\beta_{h}|_{\partial\mathcal{T}_{h}}^{2}}\\
 &\hspace{4mm}+
 \displaystyle{2\tau\|\bm{w}_{h}\|_{h,\kappa^{-1}}^{2}+\frac{5\tau^2}{2\delta}\|\nabla \cdot \bm{w}_{h}\|_{h}^{2}
 +4\tau^{2} h\|\bm{w}_{h}\cdot \bm{n}_{e}\|_{h}^{2}}\\
&\leq \gamma_{2}^{2}\vertiii{(\bm{w}_{h},p_{h},\beta_{h},\bm{u}_{h})}^2,
\end{align*}
where $\displaystyle{\gamma_{2}^{2}=\max\{2+\frac{\gamma_{B}^{2}}{2}, \frac{2}{C}+\frac{\gamma_{B}^{2}}{2C}+\frac{2}{C^{2}}+2,4\}}$. Therefore,
the inf-sup condition holds with $\displaystyle{\gamma_{\ast}=\frac{\gamma_{1}}{\gamma_{2}}}$.

 For the boundedness of $\mathcal{B}(\cdot,\cdot)$, using the Cauchy-Schwarz inequality for all the terms on right hand side and $\displaystyle{\|\nabla \cdot \bm{v}\|\leq \frac{1}{\zeta}\| \bm{v}\|_{a}}$, it is easy to show that $\mathcal{B}(\cdot,\cdot)$ is continuous with respect to norm $\vertiii{\cdot}$.
 \end{proof}

\subsection{ Error Estimates}
To derive the error analysis of the fully discrete scheme, following the error analysis in~\cite{F. Brezzi,C. Rodrigo,V. Thomee}, we first define the following elliptic projections $\overline{\bm{u}}_{h}\in \bm{V}_{h}$, $\overline{p}_{h}\in Q_{h}$, $\overline{\beta}_{h}\in B_{h}$ and $\overline{\bm{w}}_{h}\in \bm{W}_{h}$ for $t>0$,
\begin{equation}\label{e:elliptic projections}
\begin{array}{l}
a(\overline{\bm{u}}_{h},\bm{v}_{h})- \alpha(\overline{p}_{h},\nabla \cdot\bm{v}_{h}) = a(\bm{u},\bm{v}_{h})-\alpha (p,\nabla \cdot\bm{v}_{h}) ,\ \forall\ \bm{v}_{h} \in \bm{V}_{h},\\
(\nabla \cdot \overline{\bm{w}}_{h},q_{h}) =(\nabla \cdot \bm{w},q_{h}),\ \forall\ q_{h} \in Q_{h},\\
(\overline{\bm{w}}_{h}\cdot \bm{n}_{e},\rho_{h})_{\partial\mathcal{T}_{h}} = (\bm{w}\cdot \bm{n}_{e},\rho_{h})_{\partial\mathcal{T}_{h}},\  \forall\ \rho_{h} \in B_{h},\\
(\kappa^{-1}\overline{\bm{w}}_{h},\bm{r}_{h})_{h}-(\overline{p}_{h},\nabla \cdot \bm{r}_{h})-( \overline{\beta}_{h},\bm{r}_{h}\cdot \bm{n}_{e})_{\partial\mathcal{T}_{h}}\\=(\kappa^{-1}\bm{w},\bm{r}_{h})_{h}-(p,\nabla \cdot \bm{r}_{h})-( \beta,\bm{r}_{h}\cdot \bm{n}_{e})_{\partial\mathcal{T}_{h}}, \ \forall\ \bm{r}_{h} \in\bm{W}_{h}.
\end{array}
\end{equation}
We split the errors as follows,
\begin{equation*}
\begin{array}{l}
\bm{u}(t_{n})-\bm{u}_{h}^{n}=\bm{u}(t_{n})-\overline{\bm{u}}_{h}^{n}-(\bm{u}_{h}^{n}-\overline{\bm{u}}_{h}^{n})=: \rho_{\bm{u}}^{n}-e_{\bm{u}}^{n},\\
p(t_{n})-p_{h}^{n}=p(t_{n})-\overline{p}_{h}^{n}-(p_{h}^{n}-\overline{p}_{h}^{n})=:\rho_{p}^{n}-e_{p}^{n},\\
\beta(t_{n})-\beta_{h}^{n}=\beta(t_{n})-\overline{\beta}_{h}^{n}-(\beta_{h}^{n}-\overline{\beta}_{h}^{n})=: \rho_{\beta}^{n}- e_{\beta}^{n},\\
\bm{w}(t_{n})-\bm{w}_{h}^{n}= \bm{w}(t_{n})-\overline{\bm{w}}_{h}^{n}-(\bm{w}_{h}^{n}-\overline{\bm{w}}_{h}^{n})=:\rho_{\bm{w}}^{n}- e_{\bm{w}}^{n}.
\end{array}
\end{equation*}
And the following lemma shows the error estimates for the elliptic projection defined in (\ref{e:elliptic projections}).
\begin{lem}\label{lem:projection property} The following error estimates for the elliptic projections defined in (\ref{e:elliptic projections}) hold for $t>0$,
\begin{equation}\label{e:projection property}
\begin{array}{l}
\|\rho_{\bm{u}}\|_{1}\leq c h (\|\bm{u}\|_{2}+\|p\|_{1}),\\
\|\rho_{\bm{w}}\|\leq c h \| \bm{w}\|_{1},\\
\|\rho_{p}\| \leq c h (\|p\|_{1}+ \| \bm{w}\|_{1}).\\
\end{array}
\end{equation}
\end{lem}
\begin{proof}
The proof is same as the proof of Lemma 4.4~\cite{C. Rodrigo}, based on the error analysis of the hybrid mixed formulation of Poisson problems.
\end{proof}
Similarly, we define the elliptic projections, $\overline{\partial _{t}\bm{u}_{h}}$, $\overline{\partial _{t}p_{h}}$ and $\overline{\partial _{t}\bm{w}_{h}}$ of $\overline{\partial} _{t}\bm{u}_{h}$, $\overline{\partial} _{t}p_{h}$ and $\overline{\partial} _{t}\bm{w}_{h}$ respectively. This gives
similar estimates as above for $\partial _{t}\rho_{\bm{u}}$, $\partial _{t}\rho_{p}$ and $\partial _{t}\rho_{\bm{w}}$, where on the right-hand side of the inequalities we use norms of $\overline{\partial} _{t}\bm{u}_{h}$, $\overline{\partial} _{t}p_{h}$ and $\overline{\partial} _{t}\bm{w}_{h}$ instead of the norms of $ \bm{u}_{h}$, $p_{h}$ and $\bm{w}_{h}$ respectively.
Then, we estimate the errors using the following norm,
\begin{equation*}\label{e:tauh norm}
\begin{array}{l}
\|(\bm{u}_{h},p_{h},\bm{w}_{h})\|_{\tau,h}:=\biggl[\displaystyle{\|\bm{u}_{h}\|_{1}^{2}
+\biggl(1+\frac{1}{M}\biggr)\|p_{h}\|^{2}+\tau\|\bm{w}_{h}\|^{2}_{h,\kappa^{-1}}}\biggr]^{1/2},
\end{array}
\end{equation*}
where $\|\bm{w}_{h}\|^{2}_{h,\kappa^{-1}}:=(\kappa^{-1}\bm{w}_{h},\bm{w}_{h})_{h}.$

\begin{thm}\label{thm:error}
Let $\bm{u}$, $p$ and $\bm{w}$ be the solution of (\ref{e:MODEL}), and $\bm{u}_{h}$, $p_{h}$ and $\bm{w}_{h}$ be the solution of (\ref{e:stableHMFE}). If they satisfy the following regularity assumptions,
\begin{equation*}
\begin{array}{l}
\bm{u}\in L^{\infty}((0,T],\bm{H}^{1}_{0}(\Omega))\cap L^{\infty}((0,T],\bm{H}^{2}(\Omega)),\\ \partial_{t}\bm{u}\in L^{1}((0,T],\bm{H}^{2}(\Omega)),\ \partial_{tt}\bm{u}\in L^{1}((0,T],\bm{H}^{1}(\Omega)),\\
p\in L^{\infty}((0,T],H^{1}_{0}(\Omega)),\ \partial_{t}p\in L^{1}((0,T],H^{1}(\Omega)),\ \partial_{tt}p\in L^{1}((0,T],L^{2}(\Omega)),\\
\bm{w}\in L^{\infty}((0,T],H_{0}(\mathrm{div},\Omega))\cap L^{\infty}((0,T],\bm{H}^{1}(\Omega)),\ \partial_{t} \bm{w}\in L^{1}((0,T],\bm{H}^{1}(\Omega)),
\end{array}
\end{equation*}
then,
\begin{equation}\label{e:error}
\begin{array}{l}
\|\bm{u}(t_{n})-\bm{u}_{h}^{n},p(t_{n})-p_{h}^{n},\bm{w}(t_{n})-\bm{w}_{h}^{n}\|_{\tau,h}\\
\leq
c\biggl[\displaystyle{\|e_{\bm{u}}^{0}\|_{1}+\frac{1}{M}\|e_{p}^{0}\|+\tau\int_{0}^{t_{N}}(\|\partial_{tt}\bm{u}\|_{1}
+\frac{1}{M}\|\partial_{tt}p\|)dt}\biggr]\\
\hspace{4mm}+ch\biggl[\displaystyle{\|\bm{u}\|_{2}+(1+\frac{1}{M})^{1/2}\|p\|_{1}+\|\bm{w}\|_{1}}
+\tau^{1/2}\|\bm{w}\|_{1}\\
\hspace{4mm}\displaystyle{+\int_{0}^{t_{N}}(\|\partial_{t}\bm{u}\|_{2}
+(1+\frac{1}{M})\|\partial_{t}p\|_{1}+\frac{1}{M}\|\partial_{t}\bm{w}\|_{1})dt}\biggr].
\end{array}
\end{equation}
\end{thm}
\begin{proof}
See Appendix.
\end{proof}

\subsection{Perturbation of the Bilinear Form $a(\cdot, \cdot)$}
In general, using edge/face bubbles leads to a prohibitively large linear system. To resolve this, following~\cite{C. Rodrigo} we introduce
a perturbation of $a(\cdot, \cdot)$, which has a diagonal matrix representation. It is then easy to
eliminate the unknowns corresponding to the bubble functions in $\bm{V}_{b}$.

First, consider a natural decomposition of $\bm{u} \in \bm{V}_{h}$:
\begin{equation*}
\displaystyle{\bm{u}  = \bm{u}^{ l} + \bm{u} ^{b}} ,
\end{equation*}
where $\bm{u}^{ l} \in \bm{V}_{h,1} $ is the linear part and $\bm{u} ^{b} \in \bm{V}_{b}$ is the bubble part. The local bilinear form for
$T \in \mathcal{T}_{h}$, $\bm{u} \in \bm{V}_{h}$ and $\bm{v} \in \bm{V}_{h}$ is,
\begin{equation*}
a_{T}(\bm{u},\bm{v})=2\mu \int_{T}\epsilon(\bm{u}):\epsilon(\bm{v})+\lambda \int_{T}\nabla\cdot\bm{u}\nabla\cdot\bm{v}.
\end{equation*}
On each element $T$, introduce
\begin{equation*}
\displaystyle{d_{b}(\bm{u}, \bm{v})=\sum _{T \in \mathcal{T}_{h}}d_{b,T}(\bm{u}, \bm{v}) = \sum _{T \in \mathcal{T}_{h}}(d + 1)
\sum _{e \in \partial T}u_{e}v_{e}a_{T}(\bm{\Phi}_{e},\bm{\Phi}_{e})}.
\end{equation*}
Then we define a perturbed bilinear form $a^{D}(\cdot, \cdot)$ of $a(\cdot, \cdot)$ as follows,
\begin{equation*}
a^{D}(\bm{u}, \bm{v}) := d_{b}(\bm{u}^{b}, \bm{v}^{b}) + a(\bm{u}^{b}, \bm{v}^{l}) + a(\bm{u}^{l}, \bm{v}^{b}) + a(\bm{u}^{l}, \bm{v}^{l}).
\end{equation*}
As shown in~\cite{C. Rodrigo}, the bilinear forms $a(\cdot, \cdot)$ and $a^{D}(\cdot, \cdot)$ are spectrally equivalent.
\begin{lem} (\cite{C. Rodrigo} Lemma 4.3)\label{equivalence for a}
 The following inequalities hold,
 \[a(\bm{u}, \bm{u}) \leq a^{D}(\bm{u}, \bm{u}) \leq \eta a(\bm{u}, \bm{u}), \forall\ \bm{u} \in \bm{V}_{h},\]
where $\eta$ depends on the shape regularity of the mesh.
\end{lem}
Now, we consider the variational problem with diagonal bubble terms: Find $(\bm{u}_{h}^{n},p_{h}^{n},\beta_{h}^{n},\bm{w}_{h}^{n})\in \bm{V}_{h}\times Q_{h}\times B_{h}\times
\bm{W}_{h} $, such that,
\begin{equation}\label{full system eqns}
\begin{array}{l}
a^{D}(\bm{u}_{h}^{n},\bm{v}_{h})-\alpha(p_{h}^{n},\nabla \cdot \bm{v}_{h})
=(\bm{f},\bm{v}_{h}),\ \ \forall \ \bm{v}_{h} \in \bm{V}_{h},\\
\displaystyle{\alpha(\nabla \cdot\overline{\partial} _{t}\bm{u}_{h}^{n},q_{h})+\frac{1}{M}(\overline{\partial} _{t}p_{h}^{n},q_{h}) +(\nabla \cdot
\bm{w}_{h}^{n},q_{h}) = (g,q_{h})}, \ \ \forall \ q_{h} \in Q_{h},\\
 (\bm{w}_{h}^{n}\cdot \bm{n}_{e},\rho_{h})_{\partial\mathcal{T}_{h}}=0, \ \ \forall \ \rho_{h} \in B_{h},\\
-(p_{h}^{n},\nabla \cdot \bm{r}_{h})-( \beta_{h}^{n},\bm{r}_{h}\cdot \bm{n}_{e})_{\partial\mathcal{T}_{h}}+(\kappa^{-1}\bm{w}_{h}^{n},\bm{r}_{h})_{h}=0, \ \ \forall \ \bm{r}_{h} \in \bm{W}_{h}.
\end{array}
\end{equation}
To show (\ref{full system eqns}) is well-posed, we define a bilinear form with $a^{D}(\cdot,\cdot)$ replacing $a(\cdot,\cdot)$ in $\mathcal{B}(\cdot,\cdot)$,
\begin{equation*}
\begin{array}{l}
\mathcal{B}^{D}((\bm{u}_{h}^{n},p_{h}^{n},\beta_{h}^{n},\bm{w}_{h}^{n}),(\bm{v}_{h},q_{h},\rho_{h},\bm{r}_{h}))\\
:=a^{D}(\bm{u}_{h}^{n},\bm{v}_{h})-\alpha(p_{h}^{n},\nabla \cdot \bm{v}_{h})
\displaystyle{+\alpha(\nabla \cdot \overline{\partial} _{t} \bm{u}_{h}^{n},q_{h})+\frac{1}{M}(\overline{\partial} _{t}p_{h}^{n},q_{h})+(\nabla \cdot\bm{w}_{h}^{n}, q_{h})}\\
\hspace{4mm}+(\bm{w}_{h}^{n}\cdot \bm{n}_{e},\rho_{h})_{\partial\mathcal{T}_{h}}
-( p_{h}^{n},\nabla \cdot\bm{r}_{h})
-(\beta_{h}^{n},\bm{r}_{h}\cdot \bm{n}_{e})_{\partial\mathcal{T}_{h}}+(\kappa^{-1}\bm{w}_{h}^{n},\bm{r}_{h})_{h}.
\end{array}
\end{equation*}
Similarly, we have the following theorem,
\begin{thm} \label{t:stabaD}
 If the finite element spaces $Q_{h}, B_{h}, \bm{W}_{h}$  are chosen as (\ref{spacesQBW}), and $\bm{V}_{h}$ as (\ref{Vh}), the following \emph{inf-sup} condition holds
 \begin{equation*}\label{wellposednessD}
\begin{array}{l}
 \displaystyle{\sup_{(\bm{v}_{h},q_{h},\rho_{h},\bm{r}_{h})\in S_{h}}}
 \frac{\mathcal{B}^{D}((\bm{u}_{h}^{n},p_{h}^{n},\beta_{h}^{n},\bm{w}_{h}^{n}),(\bm{v}_{h},q_{h},\rho_{h},\bm{r}_{h})) }{\vertiii{(\bm{v}_{h},q_{h},\rho_{h},\bm{r}_{h})}}
\geq  \gamma_{\ast\ast}\vertiii{(\bm{u}_{h},p_{h},\beta_{h},\bm{w}_{h})}, \ \
\end{array}
\end{equation*}
and $\mathcal{B}^{D}(\cdot , \cdot)$ is continuous with respect to $\vertiii{\cdot}$, i.e.,
\begin{equation*}\label{continuousaD}
\begin{array}{l}
 |\mathcal{B}^{D}((\bm{u}_{h}^{n},p_{h}^{n},\beta_{h}^{n},\bm{w}_{h}^{n}),(\bm{v}_{h},q_{h},\rho_{h},\bm{r}_{h}))|\leq \gamma^{\ast\ast}
\vertiii{(\bm{u}_{h}^{n},p_{h}^{n},\beta_{h}^{n},\bm{w}_{h}^{n})}\vertiii{(\bm{v}_{h},q_{h},\rho_{h},\bm{r}_{h})}. \ \
\end{array}
\end{equation*}
Here, the constants $\gamma_{\ast\ast}$, $ \gamma^{\ast\ast}> 0$ are independent of mesh size $h$, time step size $\tau$, and the physical parameters.
\end{thm}
\begin{proof}
Based on the equivalence between $a^{D}(\cdot,\cdot)$ and $a(\cdot,\cdot)$, the proof is similar to that of Theorem \ref{t:stab}.
\end{proof}

\section{Elimination of Bubbles and Darcy's Velocity}\label{sec:elimination}
 The perturbation of the bilinear form $a(\cdot,\cdot)$ allows for local elimination of the bubble functions. On the other hand, the normal component of Darcy's velocity is discontinuous across the interior edges, so the mass matrix $M_{\bm{w}}$ corresponding to Darcy's velocity is block diagonal. Therefore, both the bubble unknowns and Darcy's velocity can be eliminated by static condensation. In this section, we discuss such an elimination and the well-posedness of the resulting eliminated linear system.

\subsection{Elimination}
 The variational form (\ref{full system eqns}) can be represented in block matrix form,
\begin{equation*}\label{AD}
\begin{array}{l}
\mathcal{A}^{D}\left(
\begin{array}{ccccc}
 \bm{U}_{b} \\
 \bm{U}_{l}  \\
 P  \\
 B  \\
 \bm{W} \\
\end{array}
\right)=\bm{b},\ \
\mathcal{A}^{D}=\left(
\begin{array}{ccccc}
 D_{bb}&A_{bl} & \alpha B_{b}^{\top}& 0  &0\\
 A_{bl}^{\top}& A_{ll}&\alpha B_{l}^{\top} &0  &0\\
 -\alpha B_{b}&-\alpha B_{l} &\frac{1}{M} M_{p} &0 &-\tau B_{\bm{w}}\\
 0 &0 &0 &0&-\tau B_{\beta} \\
 0 &0& \tau B_{\bm{w}}^{\top}&\tau B_{\beta}^{\top}&\tau M_{\bm{w}}   \\
\end{array}
\right),
  \end{array}
\end{equation*}
where $ \bm{U}_{b}$, $\bm{U}_{l}$, $P$, $B$ and $\bm{W}$ are the unknown vectors for the bubble component of the displacement, the piecewise
linear component of the displacement, the pressure, the Lagrange multiplier, and Darcy's velocity, respectively. The blocks in $\mathcal{A}^{D}$ correspond to the following bilinear forms:
\[
a^{D}(\bm{u}_{h}^{ b}, \bm{v}_{h}^{b}) \rightarrow D_{bb},\  a(\bm{v}_{h}^{l},\bm{u}_{h}^{b} ) \rightarrow A_{bl},\ a(\bm{u}_{h}^{ l},\bm{v}_{h}^{l}) \rightarrow A_{ll},\ -( \nabla\cdot \bm{u}_{h}^{b},q_{h}) \rightarrow B_{b},\]
\[ -( \nabla\cdot \bm{u}_{h}^{l},q_{h}) \rightarrow B_{l},\ -(\nabla\cdot \bm{w}_{h}, q_{h}) \rightarrow B_{\bm{w}},\ -( \bm{w}_{h}\cdot\bm{n}_{e},\rho_{h} )_{\partial\mathcal{T}_{h}} \rightarrow B_{\beta},\]
 \[ (p_{h}, q_{h}) \rightarrow M_{p},\ (\kappa^{-1}\bm{w}_{h}, \bm{r}_{h})_{h} \rightarrow M_{\bm{w}}
.\]

After eliminating the unknowns corresponding to the bubbles and Darcy's velocity from $\mathcal{A}^{D}$, we arrive at a smaller size matrix as follows,
 \begin{equation}\label{eliminated matrix AE}
\begin{array}{l}
\mathcal{A}^{E}=\left(
\begin{array}{ccccc}
A_{\bm{u}}^{E}&\alpha(B_{\bm{u}}^{E})^{\top}&0\\
-\alpha B_{\bm{u}}^{E}&B_{p}^{E} &\tau B_{\bm{w}}M_{\bm{w}}^{-1}B_{\beta}^{\top} \\
0&\tau B_{\beta}M_{\bm{w}}^{-1}B_{\bm{w}}^{\top}&\tau B_{\beta}M_{\bm{w}}^{-1}B_{\beta}^{\top}\\
\end{array}
\right),
  \end{array}
\end{equation}
where $ \displaystyle{B_{p}^{E}=\frac{1}{M} M_{p}+ \alpha ^{2}B_{b}D_{bb}^{-1}B_{b}^{\top}
+\tau B_{\bm{w}}M_{\bm{w}}^{-1}B_{\bm{w}}^{\top}}$, $A_{\bm{u}}^{E}=A_{ll}- A_{bl}^{\top}D_{bb}^{-1}A_{bl}$, $B_{\bm{u}}^{E}=
B_{l}- B_{b}D_{bb}^{-1}A_{bl}$.  Note that, the size of $\mathcal{A}^{E}$ is same as the classical P1-RT0-P0 discretization.

\subsection{Well-posedness of the Eliminated System}

We have shown the well-posedness of system (\ref{e:stableHMFE}), and the well-posedness of (\ref{eliminated matrix AE}) follows directly since it is obtained
by static condensation. However, for the purpose of developing preconditioners for the linear system
$\mathcal{A}^{E}$, we show the well-posedness of (\ref{eliminated matrix AE}) explicitly with proper chosen weighted norms.

In order to show that the system (\ref{eliminated matrix AE}) is well-posed, we group the multiplier and the pressure together because the multiplier
is considered to be the trace of the pressure on the element boundaries. $\mathcal{A}^{E}$ can be rewritten as in the following two-by-two block form,
\begin{equation}\label{eliminated}
\begin{array}{l}
\mathcal{A}^{E}=\left(
\begin{array}{ccccc}
A_{\bm{u}}^{E}&\alpha(B_{\bm{u},p\beta}^{E})^{\top}\\
-\alpha B_{\bm{u},p\beta}^{E}&B_{p\beta}^{E}\\
\end{array}
\right),
\end{array}
\end{equation}
where
 \begin{equation*}
\begin{array}{l}
B_{p\beta}^{E}=\left(
\begin{array}{ccccc}
 B_{p}^{E}&\tau B_{\bm{w}}M_{\bm{w}}^{-1}B_{\beta}^{\top} \\
\tau B_{\beta}M_{\bm{w}}^{-1}B_{\bm{w}}^{\top}&\tau B_{\beta}M_{\bm{w}}^{-1}B_{\beta}^{\top}\\
\end{array}
\right),\ \
B_{\bm{u},p\beta}^{E}=\left(
\begin{array}{ccccc}
 B_{\bm{u}}^{E}\\
0\\
\end{array}
\right).
  \end{array}
\end{equation*}

In order to present the well-posedness theorem, we first give some useful lemmas.
\begin{lem}\label{SPD}
 If the finite element spaces $Q_{h}, B_{h}, \bm{W}_{h}$  are chosen as (\ref{spacesQBW}), and $\bm{V}_{h}$ as (\ref{Vh}), the matrix $B_{p\beta}^{E}$ is SPD.
\end{lem}
\begin{proof}
For any $\bm{x}=(p,\beta)^{\top}\in Q_{h}\times B_{h}$,
\begin{equation*}
\begin{array}{l}
(B_{p\beta}^{E}\bm{x},\bm{x})\\
=\|p\|^{2}_{B_{p}^{E}}+2\tau(B_{\bm{w}}M_{\bm{w}}^{-1}B_{\beta}^{\top}\beta,p)+\tau\|B_{\beta}^{\top}\beta\|^{2}_{M_{\bm{w}}^{-1}}\\
\geq  \left(
\begin{array}{ccccc}
[\frac{1}{M} \|p\|^{2}_{M_{p}}+\alpha ^{2}\|B_{b}^{\top}p\|^{2}_{D_{bb}^{-1}}]^{1/2}\\
\sqrt{\tau}\|B_{\bm{w}}^{\top}p\|_{M_{\bm{w}}^{-1}}\\
\sqrt{\tau}\|B_{\beta}^{\top}\beta\|_{M_{\bm{w}}^{-1}}\\
\end{array}
\right) ^{\top}
\mathcal{M}
\left(
\begin{array}{ccccc}
[\frac{1}{M} \|p\|^{2}_{M_{p}}+\alpha ^{2}\|B_{b}^{\top}p\|^{2}_{D_{bb}^{-1}}]^{1/2}\\
\sqrt{\tau}\|B_{\bm{w}}^{\top}p\|_{M_{\bm{w}}^{-1}}\\
\sqrt{\tau}\|B_{\beta}^{\top}\beta\|_{M_{\bm{w}}^{-1}}\\
\end{array}
\right),
\end{array}
\end{equation*}
where
\begin{equation*}
\begin{array}{l}\mathcal{M}=\left(
\begin{array}{ccccc}
1&0&0\\
0&1&-1\\
0&-1&1\\
\end{array}
\right).
\end{array}
\end{equation*}
The matrix $\mathcal{M}$ is symmetric positive semi-definite. The eigenvector corresponding to zero eigenvalue is $(0,1,1)^{\top}$, but the vector $([\frac{1}{M} \|p\|^{2}_{M_{p}}+\alpha ^{2}\|B_{b}^{\top}p\|^{2}_{D_{bb}^{-1}}]^{1/2}$,
$\sqrt{\tau}\|B_{\bm{w}}^{\top}p\|_{M_{\bm{w}}^{-1}},
\sqrt{\tau}\|B_{\beta}^{\top}\beta\|_{M_{\bm{w}}^{-1}})^{\top}$ can not be a multiple of $(0,1,1)^{\top}$, because when $\frac{1}{M} \|p\|^{2}_{M_{p}}
+\alpha ^{2}\|B_{b}^{\top}p\|^{2}_{D_{bb}^{-1}}=0$, we have $p=0$, which implies $\|B_{\bm{w}}^{\top}p\|_{M_{\bm{w}}^{-1}}=0$. Therefore, $B_{p\beta}^{E}$ is SPD.
\end{proof}
\begin{corol}
 If the finite element spaces $Q_{h}, B_{h}, \bm{W}_{h}$  are chosen as (\ref{spacesQBW}), and $\bm{V}_{h}$ as (\ref{Vh}), the matrix $A_{p\beta}^{E}$ is SPD, where
\begin{equation*}
\begin{array}{l}
A_{p\beta}^{E}=\left(
\begin{array}{ccccc}
 A_{p}^{E}&\tau B_{\bm{w}}M_{\bm{w}}^{-1}B_{\beta}^{\top} \\
\tau B_{\beta}M_{\bm{w}}^{-1}B_{\bm{w}}^{\top}&\tau B_{\beta}M_{\bm{w}}^{-1}B_{\beta}^{\top}\\
\end{array}
\right),
  \end{array}
\end{equation*}
and $\displaystyle{A_{p}^{E}=B_{p}^{E}+\frac{\alpha^{2}}{\zeta^{2}} M_{p}}$.
\end{corol}
\begin{proof}
The proof is similar to that of Lemma \ref{SPD}.
\end{proof}
Now, we introduce the weighted norm
\begin{equation}\label{norm eliminated}
\begin{array}{l}
\displaystyle{\|(\bm{u}_{h},p_{h},\beta_{h})\|_{D^{E}}
:=\biggl[\|\bm{u}_{h}\|_{A_{\bm{u}}^{E}}^{2}+\|(p_{h},\beta_{h})\|_{A_{p\beta}^{E}}^{2}\biggr]^{1/2}}.\\
\end{array}
\end{equation}
The following lemma is useful for proving of the well-posedness of the system (\ref{eliminated}) with respect to the norm (\ref{norm eliminated}).
\begin{lem}\label{lem678}
If the finite element spaces $Q_{h}, B_{h}, \bm{W}_{h}$  are chosen as (\ref{spacesQBW}), and $\bm{V}_{h}$ as (\ref{Vh}), then
\begin{equation}\label{inequalities1}
\begin{array}{l}
\|(B_{\bm{u}}^{E})^{\top}p\|^{2}_{(A_{\bm{u}}^{E})^{-1}}\geq\displaystyle{\frac{\gamma^{2}_{B}}{\eta^{2}\zeta^{2}}}\|p\|^{2}_{M_{p}}
-(D_{bb}^{-1}B_{b}^{\top}p,B_{b}^{\top}p),
\end{array}
\end{equation}
\begin{equation}\label{inequalities2}
\begin{array}{l}
\zeta^{2}\|B_{\bm{u}}^{E}\bm{v}\|^{2}_{M_{p}^{-1}}\leq \|\bm{v}\|^{2}_{A_{\bm{u}}^{E}},
\end{array}
\end{equation}
\begin{equation}\label{inequalities3}
\begin{array}{l}
\displaystyle{\|(B_{\bm{u}}^{E})^{\top}p\|_{(A_{\bm{u}}^{E})^{-1}}\leq\frac{1}{\zeta}\|p\|_{M_{p}}}.
\end{array}
\end{equation}
\end{lem}
\begin{proof}
See (4.19) and (4.22) in \cite{James} for the proofs of (\ref{inequalities1}) and (\ref{inequalities2}). We only show (\ref{inequalities3}).

\begin{equation*}
\begin{array}{l}
\|(B_{\bm{u}}^{E})^{\top}p\|^{2}_{(A_{\bm{u}}^{E})^{-1}}=(B_{\bm{u}}^{E}(A_{\bm{u}}^{E})^{-1}(B_{\bm{u}}^{E})^{\top}p,p)\\
\hspace{26mm}\leq\zeta\|B_{\bm{u}}^{E}(A_{\bm{u}}^{E})^{-1}(B_{\bm{u}}^{E})^{\top}p\|_{M_{p}^{-1}}\displaystyle{\frac{1}{\zeta}}\|p\|_{M_{p}}\\
\hspace{26mm}\leq\|(A_{\bm{u}}^{E})^{-1}(B_{\bm{u}}^{E})^{\top}p\|_{A_{\bm{u}}^{E}}\displaystyle{\frac{1}{\zeta}}\|p\|_{M_{p}}\\
\hspace{26mm}=\|(B_{\bm{u}}^{E})^{\top}p\|_{(A_{\bm{u}}^{E})^{-1}}\displaystyle{\frac{1}{\zeta}}\|p\|_{M_{p}}.
\end{array}
\end{equation*}
So, $\|(B_{\bm{u}}^{E})^{\top}p\|_{(A_{\bm{u}}^{E})^{-1}}\leq\displaystyle{\frac{1}{\zeta}}\|p\|_{M_{p}}$.
\end{proof}
\begin{thm}\label{eliminated wellposed thm}
 If the finite element spaces $Q_{h}, B_{h}, \bm{W}_{h}$  are chosen as (\ref{spacesQBW}), and $\bm{V}_{h}$ as (\ref{Vh}), then the eliminated system
  (\ref{eliminated}) satisfies the following \emph{inf-sup} condition,
\begin{equation}\label{eliminated wellposed}
\begin{array}{l}
\displaystyle{\inf_{\bm{x}\in \bm{V}_{h}\times Q_{h}\times B_{h}}}\displaystyle{ \sup_{\bm{y}\in \bm{V}_{h}\times Q_{h}\times B_{h} } }\frac{(\mathcal{A}^{E}\bm{x},\bm{y})}{\|\bm{x}\|_{ D^{E}}\|\bm{y}\|_{ D^{E}}}\geq \beta_{1},
\end{array}
\end{equation}
and the continuity condition,
\begin{equation}\label{eliminated continuity}
\begin{array}{l}
(\mathcal{A}^{E}\bm{x},\bm{y})\leq \beta_{2}{\|\bm{x}\|_{ D^{E}}\|\bm{y}\|_{ D^{E}}},\ \forall\ \bm{x}, \bm{y}\in \bm{V}_{h}\times Q_{h}\times B_{h}.
\end{array}
\end{equation}
Thus, (\ref{eliminated}) is well-posed with respect to the weighted norm (\ref{norm eliminated}).
\end{thm}
\begin{proof}
\eqref{inequalities1} gives a weak inf-sup condition,
\begin{equation*}
\displaystyle{\sup_{\bm{v}\in \bm{V}_{h}}}\frac{((B_{\bm{u}}^{E})^{\top}p,\bm{v})}{\|\bm{v}\|_{A_{\bm{u}}^{E}}}\geq
\biggl(\frac{\gamma^{2}_{B}}{\eta^{2}\zeta^{2}}\|p\|^{2}_{M_{p}}-(D_{bb}^{-1}B_{b}^{\top}p,B_{b}^{\top}p)\biggr)^{1/2}.
\end{equation*}
 Then for a given $p\in  Q_{h}$, there exists $\bm{h}\in \bm{V}_{h}$ such that $((B_{\bm{u}}^{E})^{\top}p,\bm{h})$$\geq \displaystyle{\frac{\gamma^{2}_{B}}{\eta^{2}\zeta^{2}}\|p\|^{2}_{M_{p}}
-(D_{bb}^{-1}B_{b}^{\top}p,B_{b}^{\top}p)}$ and $\|\bm{h}\|^{2}_{A_{\bm{u}}^{E}}=\displaystyle{\frac{\gamma^{2}_{B}}{\eta^{2}\zeta^{2}}\|p\|^{2}_{M_{p}}
-(D_{bb}^{-1}B_{b}^{\top}p,B_{b}^{\top}p)}$. For $\bm{x}=(\bm{u},p,\beta)^{\top}$, let $\bm{y}=(\bm{u}+\vartheta \bm{h},p,\beta)^{\top}$, then,
\begin{align*}
&(\mathcal{A}^{E}\bm{x},\bm{y})\\
&=\|\bm{u}\|^{2}_{A_{\bm{u}}^{E}}+\vartheta(A_{\bm{u}}^{E}\bm{u}, \bm{h})+\vartheta\alpha((B_{\bm{u}}^{E})^{\top}p,\bm{h})+\|(p,\beta)\|^{2}_{B_{p\beta}^{E}} \\
&\geq\displaystyle{\|\bm{u}\|^{2}_{A_{\bm{u}}^{E}}-\frac{1}{2}\|\bm{u}\|^{2}_{A_{\bm{u}}^{E}}-\displaystyle{\frac{\vartheta^{2}}{2}}\|\bm{h}\|^{2}_{A_{\bm{u}}^{E}}
+\vartheta \alpha\biggl(\frac{\gamma^{2}_{B}}{\eta^{2}\zeta^{2}}}\|p\|^{2}_{M_{p}}
-(D_{bb}^{-1}B_{b}^{\top}p,B_{b}^{\top}p)\biggr)\\
&\hspace{4mm}+\|(p,\beta)\|^{2}_{B_{p\beta}^{E}}
\\
&\geq\displaystyle{\frac{1}{2}\|\bm{u}\|^{2}_{A_{\bm{u}}^{E}}+\frac{3\gamma^{2}_{B}}{8\eta^{2}}\frac{\alpha^{2}}{\zeta^{2}}\|p\|^{2}_{M_{p}}
-\frac{3}{8}\alpha^{2}(D_{bb}^{-1}B_{b}^{\top}p,B_{b}^{\top}p)+\|(p,\beta)\|^{2}_{B_{p\beta}^{E}}}\\
&\geq\vartheta_{1}\|\bm{x}\|^{2}_{D^{E}},
\end{align*}   
where $\vartheta_{1}=\displaystyle{\min\{\frac{3\gamma^{2}_{B}}{8\eta^{2}},\frac{1}{2}\}}$, and we choose $\vartheta=\displaystyle{\frac{\alpha}{2}}$. On the other hand, by Cauchy-Schwarz inequality, we have,
\begin{align*}
\|\bm{y}\|^{2}_{D^{E}}
=\|\bm{u}+\vartheta \bm{h}\|^{2}_{A_{\bm{u}}^{E}}+\|(p,\beta)\|^{2}_{A_{p\beta}^{E}}
\leq\vartheta_{2}\|\bm{x}\|^{2}_{D^{E}},
\end{align*}
where $\vartheta_{2}=\displaystyle{\max\{\frac{1}{2}+\frac{\gamma^{2}_{B}}{2\eta^{2}},2\}}$. Therefore, (\ref{eliminated wellposed}) holds with
$\beta_{1}=\vartheta_{1}/\sqrt{\vartheta_{2}}$. Following from (\ref{inequalities2}), (\ref{inequalities3}) and the Cauchy-Schwarz inequality, we have
\begin{align*}
&(\mathcal{A}^{E}\bm{x},\bm{y})
\leq 2 \|\bm{x}\|_{D^{E}}\|\bm{y}\|_{D^{E}}.
\end{align*}  
Thus, (\ref{eliminated continuity}) holds with $\beta_{2}=2$, which concludes the proof.
\end{proof}

\section{Block Preconditioner}\label{sec:preconditioner}
 In this section, we use the well-posedness to develop block preconditioners for the linear system $\mathcal{A}^{E}$ (\ref{eliminated}).
 Following the general framework developed in~\cite{D. Loghin,K.A. Mardal}, we first consider block diagonal preconditioners (norm-equivalent
 preconditioners), then we discuss block triangular preconditioners following the framework developed in~\cite{D. Loghin,G. Starke,A. Klawonn,Y. Ma}
 for (FOV) equivalent preconditioners. We theoretically show that their performance is robust with respect to the discretization
 and physical parameters.
 \subsection{Block Diagonal Preconditioners}
Based on the framework proposed in~\cite{D. Loghin,K.A. Mardal}, a natural choice of a norm-equivalent preconditioner is the Riesz operator with respect to the inner product that induces the weighted norm (\ref{norm eliminated}). The Riesz operator for (\ref{norm eliminated}) takes the following block
diagonal form,
\begin{equation}\label{BDE}
\begin{array}{l}
\mathcal{B}_{D}^{E}=\left(
\begin{array}{ccccc}
A_{\bm{u}}^{E} &0 \\
0 & A_{p\beta}^{E}\\
\end{array}
\right)^{-1}.
  \end{array}
\end{equation}
Then, we have the following theorem on the condition number.
\begin{thm}\label{condition number1}
 If the finite element spaces $Q_{h}, B_{h}, \bm{W}_{h}$  are chosen as (\ref{spacesQBW}), and $\bm{V}_{h}$ as (\ref{Vh}), then
\begin{equation*}
\begin{array}{l}
\displaystyle{\mathcal{K}(\mathcal{B}_{D}^{E}\mathcal{A}^{E})=\mathcal{O}(1)=\displaystyle{\frac{\beta_{2}}{\beta_{1}}}}.
\end{array}
\end{equation*}
\end{thm}
\begin{proof}
This result follows from Theorem \ref{eliminated wellposed thm} and the theory of condition number in~\cite{D. Loghin,K.A. Mardal}.
\end{proof}
 The action of inverting the diagonal blocks is expensive and sometimes infeasible in practice. Thus, we use spectrally equivalent
 SPD approximations to replace the diagonal blocks, i.e.,
\begin{equation*}
\begin{array}{l}
\widehat{\mathcal{B}_{D}^{E}}=\left(
\begin{array}{ccccc}
 S_{\bm{u}}^{E}&0 \\
0 & S_{p\beta}^{E}\\
\end{array}
\right),
  \end{array}
\end{equation*}
where $S_{\bm{u}}^{E}$ and $S_{p\beta}^{E}$ are spectrally equivalent to the action of the inverse of the diagonal blocks $A_{\bm{u}}^{E}$ and $A_{p\beta}^{E}$, respectively, i.e.,
\begin{equation}\label{inexact inequality1}
\begin{array}{l}
c_{1,\bm{u}}^{E}(S_{\bm{u}}^{E}\bm{u},\bm{u})\leq((A_{\bm{u}}^{E})^{-1}\bm{u},\bm{u})\leq c_{2,\bm{u}}^{E}(S_{\bm{u}}^{E}\bm{u},\bm{u}),
\end{array}
\end{equation}
\begin{equation}\label{inexact inequality2}
\begin{array}{l}
c_{1,p\beta}^{E}(S_{p\beta}^{E}(p,\beta)^{\top},(p,\beta)^{\top})\leq((A_{p\beta}^{E})^{-1}(p,\beta)^{\top},(p,\beta)^{\top})\\
\hspace{39mm}\leq c_{2,p\beta}^{E}(S_{p\beta}^{E}(p,\beta)^{\top},(p,\beta)^{\top}),
\end{array}
\end{equation}
where the constants $c_{1,\bm{u}}^{E}$, $c_{2,\bm{u}}^{E}$, $c_{1,p\beta}^{E}$ and $c_{2,p\beta}^{E}$ are independent of discretization and physical parameters.

Similarly to Theorem \ref{condition number1}, we have the following result about the condition number.
\begin{thm}
If the finite element spaces $Q_{h}, B_{h}, \bm{W}_{h}$  are chosen as (\ref{spacesQBW}), and $\bm{V}_{h}$ as (\ref{Vh}), then
\begin{equation*}
\begin{array}{l}
\displaystyle{\mathcal{K}(\widehat{\mathcal{B}_{D}^{E}}\mathcal{A}^{E})=\mathcal{O}(1)=\frac{c_{2}}{c_{1}}},
\end{array}
\end{equation*}
where $\displaystyle{c_{1}=\min\{\frac{1}{c_{2,\bm{u}}^{E}},\frac{1}{c_{2,p\beta}^{E}}\}}$ and
$\displaystyle{c_{2}=\max\{\frac{1}{c_{1,\bm{u}}^{E}},\frac{1}{c_{1,p\beta}^{E}}\}}$.
\end{thm}

\subsection{Block Triangular Preconditioner}
Now, we consider more general preconditioners, in particular, block upper and lower triangular preconditioners for the linear system
$\mathcal{A}^{E}$. Following~\cite{D. Loghin}, based on the weighted norm (\ref{norm eliminated}) and Riesz operator (\ref{BDE}), the block lower triangular preconditioner takes the following block form,
\begin{equation}\label{BLE}
\begin{array}{l}
\mathcal{B}_{L}^{E}=\left(
\begin{array}{ccccc}
A_{\bm{u}}^{E} &0\\
-\alpha B_{\bm{u},p\beta} ^{E} & A_{p\beta}^{E}\\
\end{array}
\right)^{-1}.
  \end{array}
\end{equation}
And the inexact block lower triangular preconditioner is,
\begin{equation}\label{inexact BLE}
\begin{array}{l}
\widehat{\mathcal{B}_{L}^{E}}=\left(
\begin{array}{ccccc}
(S_{\bm{u}}^{E})^{-1} &0\\
-\alpha B_{\bm{u},p\beta}^{E} &(S_{p\beta}^{E})^{-1} \\
\end{array}
\right)^{-1}.
  \end{array}
\end{equation}
Next theorem shows that (\ref{BLE}) and $\mathcal{A}^{E}$ are FOV-equivalent and, therefore, $\mathcal{B}_{L}^{E}$ provides a preconditioner for general minimal residual (GMRES) method as suggested in~\cite{D. Loghin,G. Starke,A. Klawonn,Y. Ma}
\begin{thm}\label{thm BLE}
Assuming a shape regular mesh and the discretization described above, there exists constants $\Sigma_{L}$ and $\Upsilon_{L}$, independent of
discretization and physical parameters, such that, for any $\bm{x} = ( \bm{u},p,\beta)^{\top}\neq 0$,
\begin{equation*}
\begin{array}{l}
 \displaystyle{\frac{(\mathcal{B}_{L}^{E}\mathcal{A}^{E}\bm{x},\bm{x})_{(\mathcal{B}_{D}^{E})^{-1}}}{(\bm{x},\bm{x})_{(\mathcal{B}_{D}^{E})^{-1}}}\geq\Sigma_{L},\ \
\frac{\|\mathcal{B}_{L}^{E}\mathcal{A}^{E}\bm{x}\|_{(\mathcal{B}_{D}^{E})^{-1}}}{\|\bm{x}\|_{(\mathcal{B}_{D}^{E})^{-1}}}\leq \Upsilon_{L}}.
 \end{array}
\end{equation*}
\end{thm}
\begin{proof}
By direct computation, we have,
\begin{align*}
&((\mathcal{B}_{D}^{E})^{-1}\mathcal{B}_{L}^{E}\mathcal{A}^{E}\bm{x},\bm{x})\\
&=\|\bm{u}\|^{2}_{A_{\bm{u}}^{E}}+\alpha((B_{\bm{u},p\beta}^{E})^{\top}(p,\beta)^{\top},\bm{u})
+\alpha^{2}\|(B_{\bm{u},p\beta}^{E})^{\top}(p,\beta)^{\top}\|^{2}_{(A_{\bm{u}}^{E})^{-1}}+\|(p,\beta)\|^{2}_{B_{p\beta}^{E}}\\
&\geq  \left(
\begin{array}{ccccc}
\|\bm{u}\|_{A_{\bm{u}}^{E}}\\
\alpha\|(B_{\bm{u}}^{E})^{\top}p\|_{(A_{\bm{u}}^{E})^{-1}}\\
\|(p,\beta)\|_{B_{p\beta}^{E}}\\
\end{array}
\right) ^{\top}
\left(
\begin{array}{ccccc}
1&-\frac{1}{2}&0\\
-\frac{1}{2}&1&0\\
0&0&1\\
\end{array}
\right)
\left(
\begin{array}{ccccc}
\|\bm{u}\|_{A_{\bm{u}}^{E}}\\
\alpha\|(B_{\bm{u}}^{E})^{\top}p\|_{(A_{\bm{u}}^{E})^{-1}}\\
\|(p,\beta)\|_{B_{p\beta}^{E}}\\
\end{array}
\right).
\end{align*}
The matrix in the middle is SPD. Thus, there exists a $\varrho_{1} >0$ such that
\begin{equation*}
\begin{array}{l}
((\mathcal{B}_{D}^{E})^{-1}\mathcal{B}_{L}^{E}\mathcal{A}^{E}\bm{x},\bm{x})\\
\geq \varrho_{1}\biggl(\|\bm{u}\|^{2}_{A_{\bm{u}}^{E}}+\alpha^{2}\|(B_{\bm{u}}^{E})^{\top}p\|^{2}_{(A_{\bm{u}}^{E})^{-1}}+\|(p,\beta)\|^{2}_{B_{p\beta}^{E}}\biggr)\\
\geq \displaystyle{\varrho_{1}\biggl(\|\bm{u}\|^{2}_{A_{\bm{u}}^{E}}+\frac{\alpha^{2}}{2}\biggl(\frac{\gamma^{2}_{B}}{\eta^{2}\zeta^{2}}\|p\|_{M_{p}}^{2}
-(D_{bb}^{-1}B_{b}^{\top}p,B_{b}^{\top}p)\biggr)+\|p\|^{2}_{B_{p}^{E}}}\\
\hspace{4mm}+\biggl(\|(p,\beta)\|^{2}_{B_{p\beta}^{E}}-\|p\|^{2}_{B_{p}^{E}}\biggr)\biggr)\\
\geq \displaystyle{\varrho_{1}\biggl(\|\bm{u}\|^{2}_{A_{\bm{u}}^{E}}+\frac{1}{M}\|p\|_{M_{p}}^{2}+\frac{\gamma^{2}_{B}\alpha^{2}}{2\eta^{2}\zeta^{2}}\|p\|_{M_{p}}^{2}
+\frac{\alpha^{2}}{2}(D_{bb}^{-1}B_{b}^{\top}p,B_{b}^{\top}p)}\\
\hspace{4mm}+\tau(M_{\bm{w}}^{-1}B_{\bm{w}}^{\top}p,B_{\bm{w}}^{\top}p)\biggr)
+\varrho_{1}
\left(
\begin{array}{ccccc}
 p,&\beta\\
\end{array}
\right)\left(
\begin{array}{ccccc}
 0&\tau B_{\bm{w}}M_{\bm{w}}^{-1}B_{\beta}^{\top} \\
\tau B_{\beta}M_{\bm{w}}^{-1}B_{\bm{w}}^{\top}&\tau B_{\beta}M_{\bm{w}}^{-1}B_{\beta}^{\top}\\
\end{array}
\right)
\left(
\begin{array}{ccccc}
 p\\
\beta\\
\end{array}
\right)
\\
\geq \Sigma_{L}(\bm{x},\bm{x})_{(\mathcal{B}_{D}^{E})^{-1}},
\end{array}
\end{equation*}
where $\Sigma_{L}=\displaystyle{\varrho_{1}\min\{\frac{\gamma^{2}_{B}}{2\eta^{2}},\frac{1}{2}\}}$.
Using inequalities (\ref{inequalities2}), (\ref{inequalities3}) and the Cauchy-Schwarz inequality,
we have $\|\mathcal{B}_{L}^{E}\mathcal{A}^{E}\bm{x}\|_{(\mathcal{B}_{D}^{E})^{-1}}\leq \Upsilon_{L}\|\bm{x}\|_{(\mathcal{B}_{D}^{E})^{-1}}$ with $\Upsilon_{L}=2$, which concludes the proof.
\end{proof}

Now we prove the inexact block lower triangular preconditioner (\ref{inexact BLE}) satisfies the requirements to be an FOV-equivalent preconditioner for the $\mathcal{A}^{E}$ system as well, when the diagonal blocks are solved sufficiently accurately.
\begin{thm}\label{thm BLE inexact}
Assuming a shape regular mesh,  the discretization described above, and that (\ref{inexact inequality1}) and (\ref{inexact inequality2}) hold with $c_{2,\bm{u}}^{E}>\frac{1}{2}$, then there exists constants $\Sigma^{I}_{L}$ and $\Upsilon^{I}_{L}$, independent of discretization and physical parameters, such that, for any $\bm{x} = ( \bm{u},p,\beta)^{\top}\neq 0$,
\begin{equation*}
\begin{array}{l}
\displaystyle{\frac{(\widehat{\mathcal{B}_{L}^{E}}\mathcal{A}^{E}\bm{x},\bm{x})_{(\widehat{\mathcal{B}_{D}^{E}})^{-1}}}{(\bm{x},\bm{x})
_{(\widehat{\mathcal{B}_{D}^{E}})^{-1}}}\geq\Sigma^{I}_{L} ,\ \
\frac{\|\widehat{\mathcal{B}_{L}^{E}}\mathcal{A}^{E}\bm{x}\|_{(\mathcal{B}_{D}^{E})^{-1}}}{\|\bm{x}\|_{(\widehat{\mathcal{B}_{D}^{E}})^{-1}}}\leq \Upsilon^{I}_{L}}.
 \end{array}
\end{equation*}
\end{thm}
\begin{proof}
By direct computation, we have,
\begin{equation*}
\begin{array}{l}
((\widehat{\mathcal{B}_{D}^{E}})^{-1}\widehat{\mathcal{B}_{L}^{E}}\mathcal{A}^{E}\bm{x},\bm{x})\\
=\|\bm{u}\|^{2}_{A_{\bm{u}}^{E}}+\alpha(S_{\bm{u}}^{E}A_{\bm{u}}^{E}\bm{u},(B_{\bm{u},p\beta}^{E})^{\top}(p,\beta)^{\top})
+\|(p,\beta)\|^{2}_{B_{p\beta}^{E}}
\\
\hspace{4mm}+\alpha^{2}\|(B_{\bm{u},p\beta}^{E})^{\top}(p,\beta)^{\top}\|^{2}_{S_{\bm{u}}^{E}}\\
\geq  \left(
\begin{array}{ccccc}
\|\bm{u}\|_{A_{\bm{u}}^{E}}\\
\alpha\|(B_{\bm{u}}^{E})^{\top}p\|_{S_{\bm{u}}^{E}}\\
\|(p,\beta)\|_{B_{p\beta}^{E}}\\
\end{array}
\right) ^{\top}
\left(
\begin{array}{ccccc}
1&-\frac{1}{2}&0\\
-\frac{1}{2}&1&0\\
0&0&1\\
\end{array}
\right)
\left(
\begin{array}{ccccc}
\|\bm{u}\|_{A_{\bm{u}}^{E}}\\
\alpha\|(B_{\bm{u}}^{E})^{\top}p\|_{S_{\bm{u}}^{E}}\\
\|(p,\beta)\|_{B_{p\beta}^{E}}\\
\end{array}
\right).
\end{array}
\end{equation*}
The matrix in the middle is SPD. Thus, there exists a $\varrho_{2}>0$ such that
\begin{align*}
&((\widehat{\mathcal{B}_{D}^{E}})^{-1}\widehat{\mathcal{B}_{L}^{E}}\mathcal{A}^{E}\bm{x},\bm{x})\\
&\geq\varrho_{2}\biggl(\|\bm{u}\|^{2}_{A_{\bm{u}}^{E}}+\alpha^{2}\|(B_{\bm{u}}^{E})^{\top}p\|^{2}_{S_{\bm{u}}^{E}}
+\|(p,\beta)\|^{2}_{B_{p\beta}^{E}}\biggr)\\
&\geq\displaystyle{\varrho_{2}\biggl[\frac{1}{c_{2,\bm{u}}^{E}}\biggl(\|\bm{u}\|^{2}_{(S_{\bm{u}}^{E})^{-1}}+\frac{\gamma^{2}_{B}\alpha^{2}}{2\eta^{2}
\zeta^{2}}\|p\|^{2}-\frac{\alpha^{2}}{2}(D_{bb}^{-1}B_{b}^{\top}p,B_{b}^{\top}p)\biggr)+\|p\|^{2}_{B_{p}^{E}}}\\
&\hspace{4mm}+\biggl(\|(p,\beta)\|^{2}_{B_{p\beta}^{E}}
-\|p\|^{2}_{B_{p}^{E}}\biggr)\biggr]\\
&\geq \displaystyle{\varrho_{2}\biggl(\frac{1}{c_{2,\bm{u}}^{E}}\|\bm{u}\|^{2}_{(S_{\bm{u}}^{E})^{-1}}
+\frac{\gamma^{2}_{B}}{2c_{2,\bm{u}}^{E}\eta^{2}}\frac{\alpha^{2}}{\zeta^{2}}\|p\|^{2}
+(1-\frac{1}{2c_{2,\bm{u}}^{E}})\alpha^{2}(D_{bb}^{-1}B_{b}^{\top}p,B_{b}^{\top}p)}\\
&\hspace{4mm}\displaystyle{+\frac{1}{M}\|p\|^{2}}+\tau(M_{\bm{w}}^{-1}B_{\bm{w}}^{\top}p,B_{\bm{w}}^{\top}p)\biggr)\\
&\hspace{4mm}+
\varrho_{2}
\left(
\begin{array}{ccccc}
 p,&\beta\\
\end{array}
\right)\left(
\begin{array}{ccccc}
 0&\tau B_{\bm{w}}M_{\bm{w}}^{-1}B_{\beta}^{\top} \\
\tau B_{\beta}M_{\bm{w}}^{-1}B_{\bm{w}}^{\top}&\tau B_{\beta}M_{\bm{w}}^{-1}B_{\beta}^{\top}\\
\end{array}
\right)
\left(
\begin{array}{ccccc}
 p\\
\beta\\
\end{array}
\right)
\\
&\geq\displaystyle{\varrho_{2}\frac{1}{c_{2,\bm{u}}^{E}}\biggl(\|\bm{u}\|^{2}_{(S_{\bm{u}}^{E})^{-1}}
+\min\{\frac{\gamma^{2}_{B}}{2\eta^{2}},c_{2,\bm{u}}^{E}-\frac{1}{2}\}\|(p,\beta)\|^{2}_{A_{p\beta}^{E}}\biggr)}\\
&\geq \Sigma^{I}_{L}(\bm{x},\bm{x})_{(\widehat{\mathcal{B}_{D}^{E}})^{-1}},
\end{align*}
where $\Sigma^{I}_{L}=\displaystyle{\frac{\varrho_{2}}{2c_{2,p\beta}^{E}c_{2,\bm{u}}^{E}}\min\{ 2c_{2,p\beta}^{E}, \frac{\gamma^{2}_{B}}{\eta^{2}},
2c_{2,\bm{u}}^{E}-1\}}$. The upper bound follows from the inequalities (\ref{inequalities2}), (\ref{inequalities3}), and (\ref{inexact inequality1}), and Cauchy-Schwarz inequality,
 with $ \Upsilon^{I}_{L}=\displaystyle{\max\{\frac{3}{(c_{1,\bm{u}}^{E})},\frac{1}{(c_{1,p\beta}^{E})}(1+\frac{2}{c_{1,\bm{u}}^{E}})\}}$.
\end{proof}
Similarly, we can also consider the following block upper triangular preconditioner for $\mathcal{A}^{E}$,
\begin{equation}\label{BUE}
\begin{array}{l}
\mathcal{B}_{U}^{E}=\left(
\begin{array}{ccccc}
 A_{\bm{u}}^{E} &\alpha (B_{\bm{u},p\beta} ^{E})^{\top}\\
 &A_{p\beta}^{E}\\
\end{array}
\right)^{-1},
  \end{array}
\end{equation}
and its corresponding inexact version,
\begin{equation}\label{inexact BUE}
\begin{array}{l}
\widehat{\mathcal{B}_{U}^{E}}=\left(
\begin{array}{ccccc}
(S_{\bm{u}}^{E})^{-1} &\alpha (B_{\bm{u},p\beta}^{E})^{\top}\\
0 & (S_{p\beta}^{E})^{-1}\\
\end{array}
\right)^{-1}.
  \end{array}
\end{equation}

The following theorems show that these block upper triangular preconditioners are also parameter robust. The proofs are similar to the proofs for Theorem \ref{thm BLE} and \ref{thm BLE inexact} and, therefore, are omitted.
\begin{thm}\label{thm BUE}
Assuming a shape regular mesh and the discretization described above, then there exists constants $\Sigma_{U}$ and $\Upsilon_{U}$, independent of discretization and physical parameters, such that, for any $\bm{x}^{\prime} =(\mathcal{B}_{U}^{E})^{-1}\bm{x}$, with $\bm{x}=( \bm{u},p,\beta)^{\top}\neq 0$,
\begin{equation*}
\begin{array}{l}
\displaystyle{\frac{(\mathcal{A}^{E}\mathcal{B}_{U}^{E}\bm{x}^{\prime} ,\bm{x}^{\prime} )_{\mathcal{B}_{D}^{E}}}{(\bm{x}^{\prime},\bm{x}^{\prime})_{\mathcal{B}_{D}^{E}}}\geq\Sigma_{U},\ \
\frac{\|\mathcal{A}^{E}\mathcal{B}_{U}^{E}\bm{x}^{\prime}\|_{\mathcal{B}_{D}^{E}}}{\|\bm{x}^{\prime}\|_{\mathcal{B}_{D}^{E}}}\leq \Upsilon_{U}}.
 \end{array}
\end{equation*}
\end{thm}
\begin{thm}\label{thm BUE inexact}
Assuming a shape regular mesh, the discretization described above, and that (\ref{inexact inequality1}) and (\ref{inexact inequality2}) hold  with
$c_{2,\bm{u}}^{E}>\frac{1}{2}$, then there exists constants $\Sigma^{I}_{U}$ and $\Upsilon^{I}_{U}$, independent of discretization or physical parameters,
 such that, for any $\bm{x}^{\prime} =(\widehat{\mathcal{B}_{U}^{E}})^{-1}\bm{x}$, with $\bm{x}=( \bm{u},p,\beta)^{\top}\neq 0$,
\begin{equation*}
\begin{array}{l}
\displaystyle{\frac{(\mathcal{A}^{E}\widehat{\mathcal{B}_{U}^{E}}\bm{x}^{\prime},\bm{x}^{\prime})_{\widehat{\mathcal{B}_{D}^{E}}}}{(\bm{x}^{\prime},\bm{x}^{\prime})
_{\widehat{\mathcal{B}_{D}^{E}}}}\geq\Sigma^{I}_{U} ,\ \
\frac{\|\mathcal{A}^{E}\widehat{\mathcal{B}_{U}^{E}}\bm{x}^{\prime}\|_{\widehat{\mathcal{B}_{D}^{E}}}}{\|\bm{x}^{\prime}\|_{\widehat{\mathcal{B}_{D}^{E}}}}\leq \Upsilon^{I}_{U}}.
 \end{array}
\end{equation*}
\end{thm}

\section{Numerical Experiments}\label{sec:numerical}
In this section, we give some numerical examples. The examples in Section 6.1 are used to investigate the accuracy of the stabilized hybrid method.
In addition, the cantilever bracket problem is proposed to show the effectiveness of our discretization. In Section 6.2, we demonstrate
the robustness of the preconditioners presented in Section 5. In all test cases we consider a diagonal permeability tensor $\kappa=K\bm{I}$ with constant $K$.
\subsection{The Accuracy and Efficiency of the Stabilization Method}\label{sec:numerical1}
Firstly, we demonstrate numerically that for Biot's model, the error in the finite element approximation does not decrease when the permeability is small relative to the mesh size. We consider $\Omega=(0,1)\times(0,1)$ with homogeneous Dirichlet boundary conditions for $\bm{u}$, and Neumann boundary conditions for $p$. We cover $\Omega$ with a uniform triangular grid by dividing an $N \times N$ uniform square mesh into right
triangles, where the mesh spacing is defined by $h=1/N$. The material parameters are $\lambda = 2$, $\mu = 1$, $\alpha = 1$,
and $M = 10^{6}$. We consider a diagonal permeability tensor $\kappa=K \bm{I}$ with constant $K$. The data is set so that the exact solution is given by
\begin{equation*}
p=1,\ \ \bm{u}= \nabla \times \varphi=
\left(
\begin{array}{ccc}
 \partial _{y} \varphi\\
-\partial _{x} \varphi
\end{array}
\right),\ \ \varphi(x,y)=[x y(1-x)(1-y)]^{2}.
\end{equation*}
Note that the solution is designed to satisfy $\nabla\cdot \bm{u}=0$ at any time $t$. We set $\tau = 1$ and $t_{max} = 1$.

 As shown in Table \ref{table:1}, the results for the energy norm ($\|\bm{v}\|^{2}_{a} := a(\bm{v}, \bm{v})$) of the displacement errors and the $L^{2}$-norm for pressure errors convergence when $K$ is relatively big. When $K$ becomes smaller, the performance becomes worse, and eventually
  divergences when $K$ is really small.
\begin{table}[!h]
\begin{center}
\caption{Error and convergence rate of the hybrid mixed finite element method~\eqref{e:HMFE}.}\label{table:1}
\begin{tabular}{ccccccc}
\hline\noalign{\smallskip}
 &  &$N=4$&  $N=8$ &  $N=16$  & $N=32$  &$N=64$\\
\noalign{\smallskip}\hline\noalign{\smallskip}
\textit{K}=1e-4&$ \|\bm{u}-\bm{u}_{h}\|_{a}$         &   0.0509  & 0.0270   &  0.0135  &  0.0068  &  0.0034 \\
      &$ \|p-p_{h}\|$                           &   0.1160  &  0.0535  &  0.0088  &  0.0015  &  0.0003 \\
\noalign{\smallskip}\hline\noalign{\smallskip}
\textit{K}=1e-6&$ \|\bm{u}-\bm{u}_{h}\|_{a}$         & 0.0570    & 0.0543  &  0.0314  &  0.0081  &  0.0034  \\
   &$ \|p-p_{h}\|$                              &  0.1587   &  0.3277 &   0.3199  &  0.0763  &  0.0099  \\
\noalign{\smallskip}\hline\noalign{\smallskip}
\textit{K}=1e-8&$ \|\bm{u}-\bm{u}_{h}\|_{a}$         &  0.0571  &  0.0571  &  0.0565  &  0.0478  &  0.0169  \\
      &$ \|p-p_{h}\|$                           &   0.1591   & 0.3553  &  0.7157  &  1.1509  &  0.6537     \\
\noalign{\smallskip}\hline\noalign{\smallskip}
\textit{K}=1e-10&$ \|\bm{u}-\bm{u}_{h}\|_{a}$        &    0.0571  &  0.0571  &  0.0571  &  0.0570  &  0.0550 \\
       &$ \|p-p_{h}\|$                          &   0.1588  &  0.3550   & 0.7271  &  1.4616  &  2.9182 \\
\noalign{\smallskip}\hline\noalign{\smallskip}
\end{tabular}
\end{center}
\end{table}
 Then, we use the proposed stabilized hybrid discretization to solve it. Table \ref{tab: error results for stable scheme} shows good convergence results even when $K \rightarrow 0$.

We also compare the errors obtained by (\ref{full system eqns}) with diagonal bubble functions with those provided by the bubble enriched system (\ref{e:stableHMFE}), in order to see that the same error reduction is achieved. Fig. \ref{Fig:compare a and aD} displays a comparison of the displacement and pressure errors in
the energy and $L^{2}$ norms, respectively, for different grid sizes. We choose $K= 10^{-8}$ here, though similar results are obtained for different
values of $K$. We observe both schemes produces the same convergence rate  although the scheme corresponding to the diagonal version provides
slightly worse errors. However, this scheme, when the bubble block is eliminated, uses fewer degrees of freedom, thus computationally more efficient.
\begin{table}[!h]
\begin{center}
\caption{ Error and convergence rate of the stabilized hybrid mixed finite element method~\eqref{full system eqns} with perturbation}\label{tab: error results for stable scheme}
\begin{tabular}{ccccccc}
\hline\noalign{\smallskip}
 &  &$N=4$&  $N=8$ &  $N=16$  & $N=32$  &$N=64$ \\
      \noalign{\smallskip}\hline\noalign{\smallskip}
\emph{K}=1e-4&$ \|\bm{u}-\bm{u}_{h}\|_{a}$         & 0.0369&   0.0183&   0.0093&   0.0047 &  0.0024\\
      &$ \|p-p_{h}\|$                   & 0.0511&   0.0185&   0.0034&   0.0006 &  0.0001  \\
      \noalign{\smallskip}\hline\noalign{\smallskip}
\emph{K}=1e-6&$ \|\bm{u}-\bm{u}_{h}\|_{a}$         & 0.0377&   0.0189&   0.0091&   0.0045 &  0.0022  \\
      &$ \|p-p_{h}\|$                   & 0.0593&   0.0346&   0.0155&   0.0062 &  0.0019  \\
      \noalign{\smallskip}\hline\noalign{\smallskip}
\emph{K}=1e-8&$ \|\bm{u}-\bm{u}_{h}\|_{a}$         & 0.0377&   0.0189&   0.0092&   0.0045 &  0.0023  \\
      &$ \|p-p_{h}\|$                   & 0.0594&   0.0349&   0.0162&   0.0074 &  0.0035  \\
   \noalign{\smallskip}\hline\noalign{\smallskip}
\emph{K}=1e-10&$ \|\bm{u}-\bm{u}_{h}\|_{a}$        & 0.0377&   0.0189&   0.0092&   0.0045 &  0.0023  \\
       &$ \|p-p_{h}\|$                  & 0.0594&   0.0349&   0.0162&   0.0074 &  0.0035  \\
\noalign{\smallskip}\hline\noalign{\smallskip}
\end{tabular}
\end{center}
\end{table}
\begin{figure}[!htb]
    \centering
    \subfloat[]{\label{Fig:displacementerr}
    \includegraphics[width=2.2in]{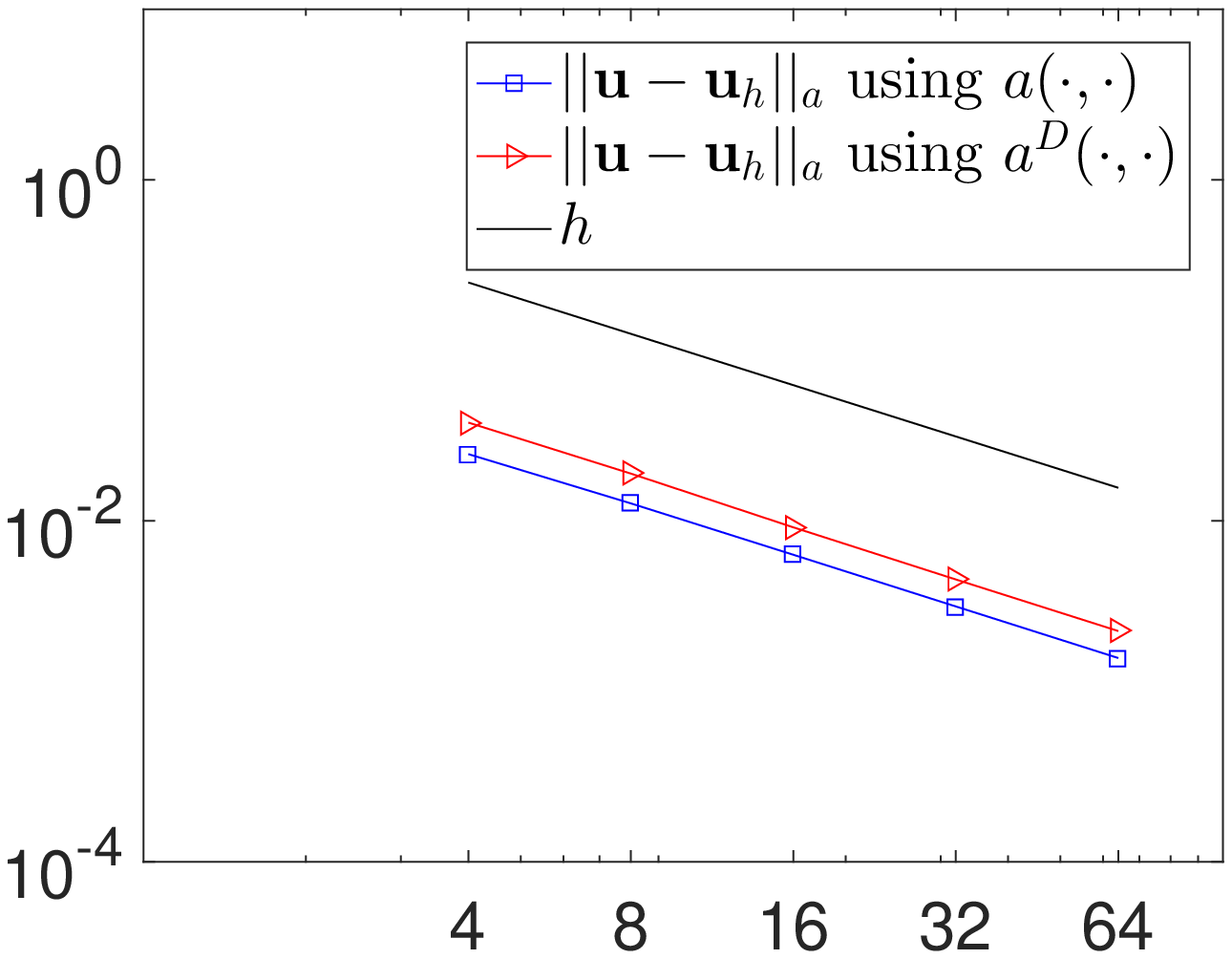}}
    \quad
    \subfloat[]{\label{Fig:pressureerr}
    \includegraphics[width=2.2in]{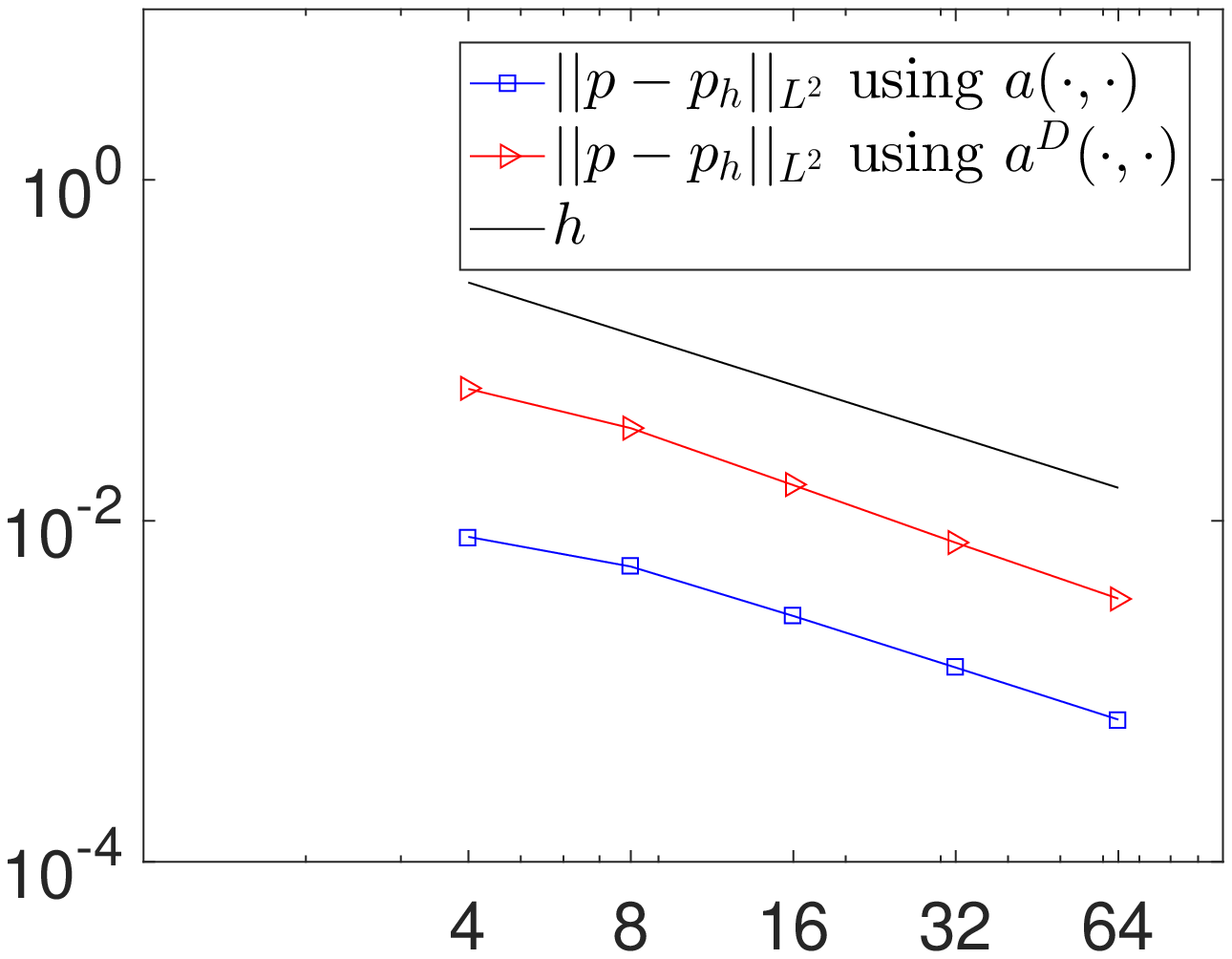}}\\
    \caption{Reduction of the (a) displacement and (b) pressure errors for different mesh-sizes, by using the enriched
finite element scheme~(\ref{e:stableHMFE}), as well as the scheme (\ref{full system eqns}) with perturbations.}\label{Fig:compare a and aD}
\end{figure}

We also consider the cantilever bracket problem~\cite{P. J. Phillips4} illustrated in Fig. \ref{Fig:12}. The domain is the unit square $[0, 1] \times [0, 1]$. For the flow problem we impose a no-flow boundary condition along the entire boundary.
 For the elasticity problem, we assume that the left side boundary of the domain is  fixed, so a no-displacement boundary condition is imposed. We
 impose a downward traction at the top of the domain and a traction-free boundary condition at the right and bottom of the domain. The initial
 displacement and pressure are assumed to be zero.
\begin{figure}[h!]
  \centering
  \includegraphics[height=170pt,width=210pt]{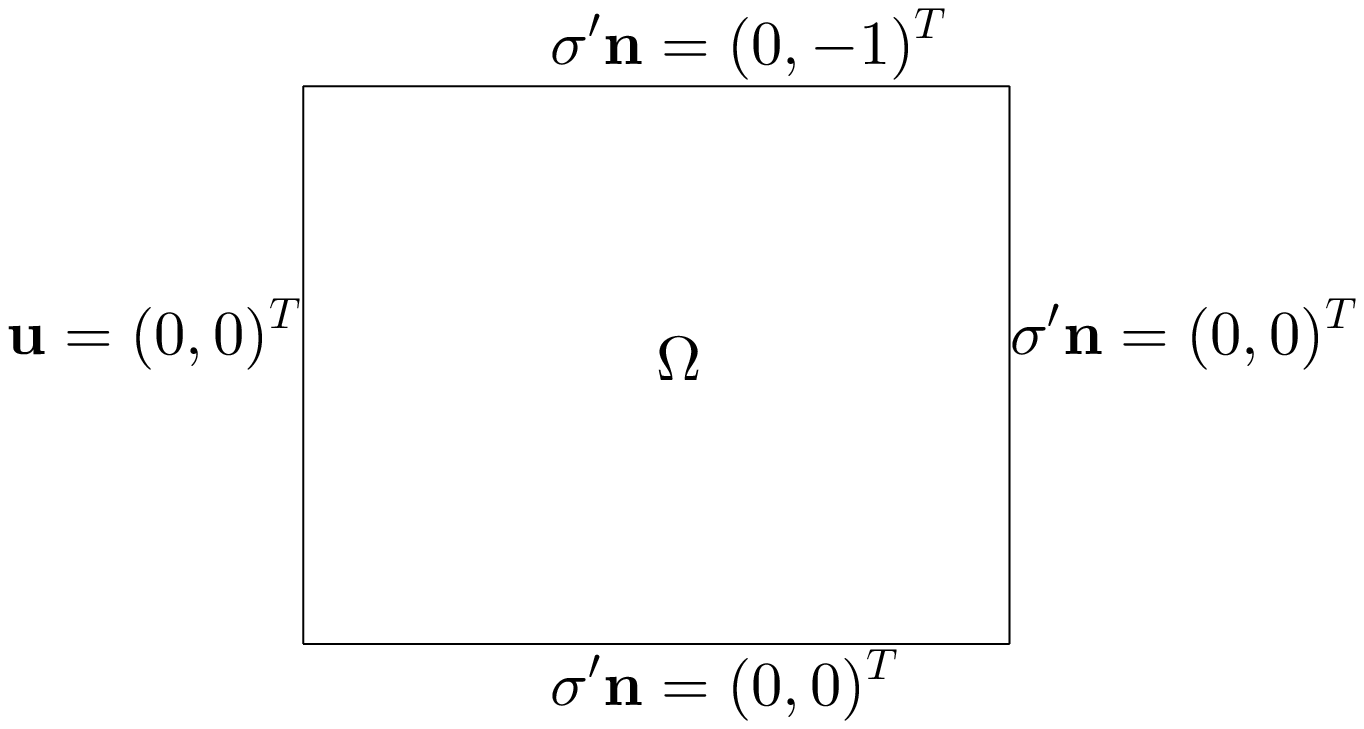}
  \caption{ Description of the cantilever bracket problem.}
  \label{Fig:12}
\end{figure}
For the material properties, the Lam\'{e} coefficients are computed in terms of the Young modulus, $E$, and the Poisson ratio $\nu$:
$\lambda=\frac{E \nu}{(1-2 \nu)(1+\nu)}$, and $\mu=\frac{E}{1+2 \nu}$ with $\nu=0.45$, $E=10^{5}$, and $K=10^{-7},\ \alpha= 0.93,\ M=10^{10}$.
Setting $ \tau =0.001$, we show the approximation for the pressure field obtained without stabilization terms in Fig. \ref{Fig:pressure-CBproblem} (left),
and observe nonphysical oscillations. We clearly observe that the stability scheme (\ref{eliminated}) removes the nonphysical pressure oscillations, see
Fig. \ref{Fig:pressure-CBproblem} (right).
\begin{figure}[!htb]
    \centering
   \subfloat[]{\label{Fig:Poscillation}
    \includegraphics[width=2in]{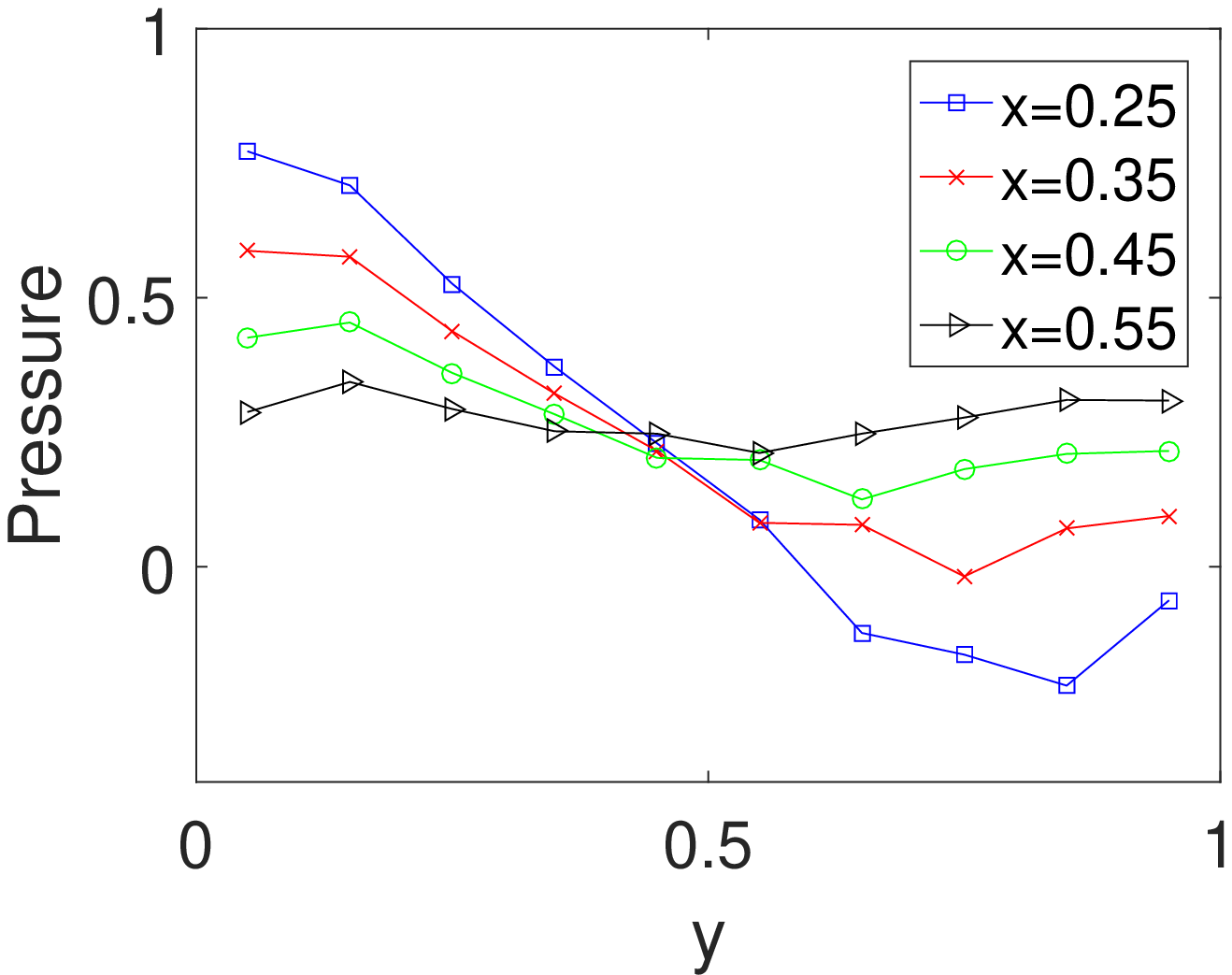}}
    \quad
    \subfloat[]{\label{Fig:Psmooth}
    \includegraphics[width=2in]{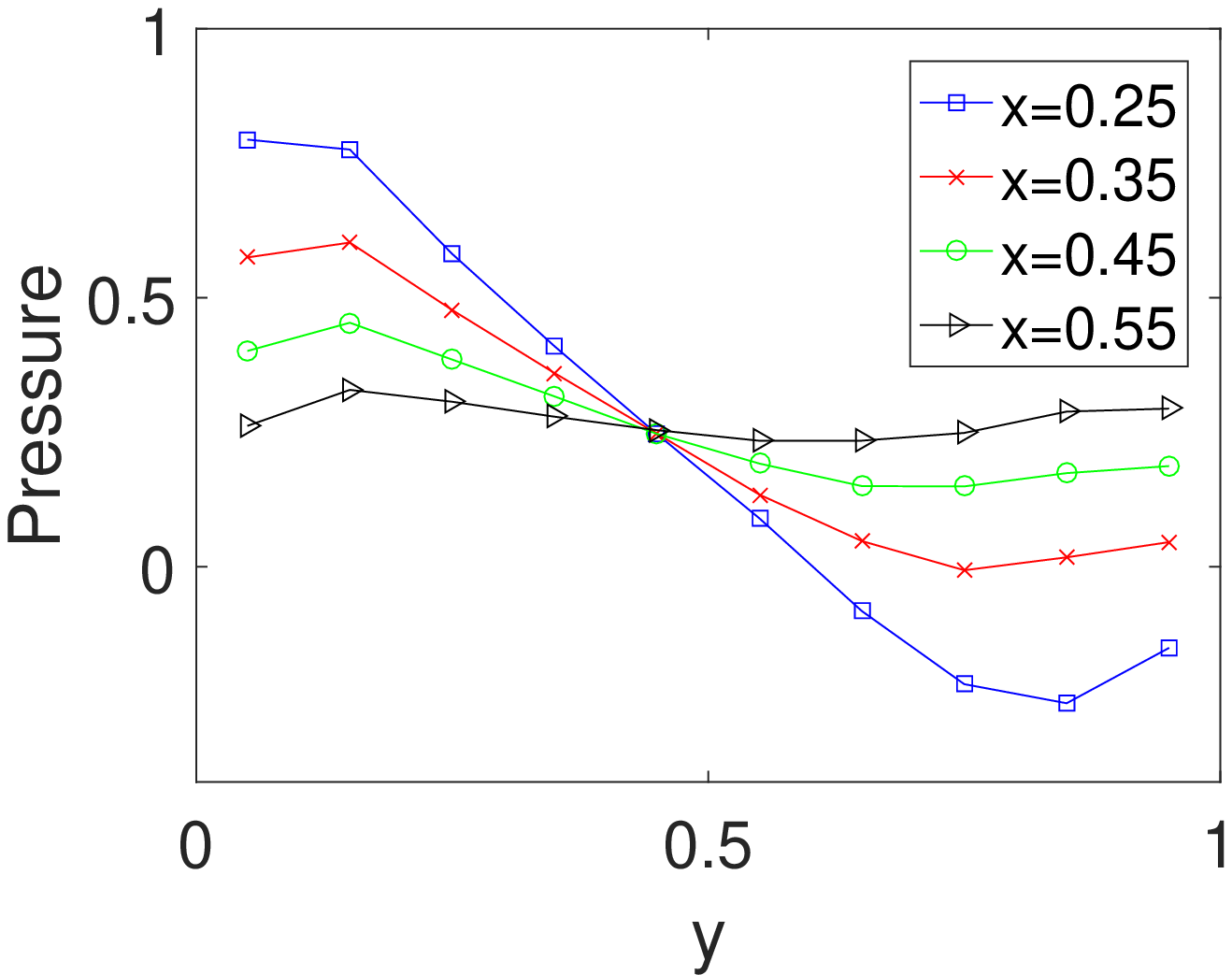}}\\
    \caption{Pressure solution of the cantilever bracket problem using hybrid P1-RT0-P0 (left) and eliminated system (\ref{eliminated}) (right) finite element method at $t_{max}$ = 0.005 (final time).}\label{Fig:pressure-CBproblem}
\end{figure}

\subsection{Robustness of the Preconditioners}
 In this section, we demonstrate the robustness of the block preconditioners. We use flexible GMRES to solve the eliminated system $\mathcal{A}^{E}$ (\ref{eliminated}). For each test we use flexible GMRES to solve the linear system $\mathcal{A}^{E}$ and a stopping tolerance of $10^{-8}$ for the relative residual is used.  Different numerical tests are presented to show the robustness of the proposed preconditioners with respect to the physical parameters (mainly the permeability, $\kappa$ and the poisson ratio $\nu$) and the discretization parameters (mesh size $h$ and time step size $\tau$).  In the exact block preconditioners $\mathcal{B}_{D}^{E}$, $\mathcal{B}_{L}^{E}$, and $\mathcal{B}_{U}^{E}$, the diagonal blocks are inverted by the UMFPACK library~\cite{T. A. Davis1,T. A. Davis2,T. A. Davis3,T. A. Davis04}. For the inexact block preconditioners $\widehat{\mathcal{B}_D^E}$, $\widehat{\mathcal{B}_L^E}$, and $\widehat{\mathcal{B}_U^E}$,  $S_{\bm{u}}^{E}$ and $S_{p,\beta}^{E}$ are implemented using conjugate gradient (CG) method preconditioned with a V-cycle unsmoothed aggregation AMG (UA-AMG) and solved to a tolerance that the relative residual is less than $10^{-3}$.  We use zero right hand side and random initial guess in all the tests.  Each test is repeated $5$ times and the average number of iterations are reported.

We consider the first example presented in Section~\ref{sec:numerical1}. Choosing the physical parameters as $M=10^{6}$, $\alpha=1$, $E=1$, Table~\ref{tab:prec-elim-2d-nu-K1} shows iteration counts for the block preconditioners for the eliminated system when the physical values of $\nu$ and $K$ are varying. The mesh size is fixed as $h=\frac{1}{64}$, and the time step size is $\tau=1$. When $K$ is varying, the number of iterations grows slightly at the beginning but stabilizes when $K$ approached $10^{-12}$ for block diagonal preconditioners $\mathcal{B}_D^E$ and $\widehat{\mathcal{B}_D^E}$.  Block triangular preconditioners perform better than the block diagonal preconditioners as expected and are stable with respect to $K$. Overall, from the relatively consistent iteration counts, we can conclude that the block preconditioners are indeed robust with respect to the physical parameters. Table~\ref{tab:prec-elim-2d-h-tau1} shows the number of iterations for the block preconditioners for solving $\mathcal{A}^{E}$ for different mesh sizes and time step sizes with fixed $ \lambda=2$, $\mu=1$, and $K=10^{-6}$. Again, we observe robustness with respect to the discretization parameters. The block upper and lower triangular preconditioners contain more coupling information than the block diagonal preconditioners, and, as a result, we can see that they perform better than the block diagonal preconditioners. In addition, from both tables, we can see that the number of iterations of the inexact block preconditioners is slightly higher than that of the exact block preconditioners. This is expected and the difference is by no means significant.  Therefore, we would suggest using the inexact version in practice to avoid expensive direct solves used in the exact version.

\begin{table}[htp]
\footnotesize
\begin{center}
\caption{Iteration counts of the block preconditioners for the eliminated system $\mathcal{A}^{E}$ discretized from the first example in Section \ref{sec:numerical1}  with varying physical parameters $K$ and $\nu$.} \label{tab:prec-elim-2d-nu-K1}
\begin{tabular}{c|c c c c c c}
  \hline
  & \multicolumn{6}{ c }{$\nu = 0$ and varying $K$ 
  } \\ \hline
 & $10^{-2}$ & $10^{-4}$ & $10^{-6}$ & $10^{-8}$ & $10^{-10}$& $10^{-12}$  \\ \hline
	$\mathcal{B}_D^E$           & 21  & 28 & 38  & 40  & 40 & 38 \\
	$\mathcal{B}_U^E$           & 12  & 13 & 14  & 15  & 15 & 15  \\
	$\mathcal{B}_L^E$           & 13  & 14 & 14  & 15  & 15 & 15  \\
	$\widehat{\mathcal{B}_D^E}$ & 29  & 38 & 44  & 46  & 44 & 43 \\
	$\widehat{\mathcal{B}_U^E}$ & 16  & 18 & 23  & 22  & 23 & 21 \\
	$\widehat{\mathcal{B}_L^E}$ & 20  & 22 & 21  & 22  & 20 & 20 \\

  \hline
  & \multicolumn{6}{ c }{$K = 10^{-6}$ and varying $\nu$} \\ \hline
  & $0$ & $0.1$ & $0.2$ & $0.4$ & $0.45$ & $0.49$   \\ \hline
	$\mathcal{B}_D^E$           &38  & 38  & 38  & 36 & 33  & 29  \\
	$\mathcal{B}_U^E$           &14  & 14  & 14  & 13 & 11  & 8   \\
	$\mathcal{B}_L^E$           &14  & 14  & 14  & 13 & 11  & 8   \\
	$\widehat{\mathcal{B}_D^E}$ &44  & 44  & 44  & 45 & 44  & 40 \\
	$\widehat{\mathcal{B}_U^E}$ &23  & 21  & 20  & 19 & 15  & 12   \\
	$\widehat{\mathcal{B}_L^E}$ &21  & 21  & 20  & 16 & 15  & 12   \\

  \hline
\end{tabular}
\end{center}
\end{table}

\begin{table}[h!]
\footnotesize
\begin{center}
\caption{Iteration counts of the block preconditioners for the eliminated system $\mathcal{A}^{E}$ discretized from the first example in Section~\ref{sec:numerical1} with varying discretization parameters $h$ and $\tau$.}\label{tab:prec-elim-2d-h-tau1}
\begin{tabular}{l | c c c c c c c c c c}
	\hline
&\multicolumn{5}{ c }{$\mathcal{B}_D^E $}\\[0.3em] \hline
		 \backslashbox{$\tau$}{$h$} & $\frac{1}{4}$ & $ \frac{1}{8}$ & $\frac{1}{16}$  & $\frac{1}{32}$ & $\frac{1}{64}$  \\
		\hline
        $1$      &35 & 39 & 40 & 38 & 35 \\
		$0.1$    &32 & 39 & 40 & 40 & 39 \\
		$0.01$   &34 & 39 & 40 & 40 & 39 \\
		$0.001$  &33 & 37 & 38 & 38 & 38 \\
		$0.0001$ &33 & 38 & 38 & 38 & 38 \\
		\hline
		\end{tabular}
\begin{tabular}{c c c c c}
        \hline
\multicolumn{5}{ c }{$\widehat{\mathcal{B}_D^E}  $}\\ \hline
$\frac{1}{4}$ \hspace{-58pt}\phantom{\backslashbox{$\tau$}{Mesh}}  & $\frac{1}{8}$ & $\frac{1}{16}$  & $\frac{1}{32}$ & $\frac{1}{64}$ \\
		\hline
        39 & 44 & 46 & 47 & 46 \\
		39 & 44 & 45 & 45 & 45 \\
		38 & 44 & 45 & 46 & 43 \\
		38 & 45 & 45 & 44 & 43 \\
		39 & 44 & 45 & 45 & 43 \\
		\hline
		\end{tabular}
\begin{tabular}{l | c c c c c}
		\hline
&\multicolumn{5}{ c }{$\mathcal{B}_U^E$}\\[0.3em] \hline
		 \backslashbox{$\tau$}{$h$} & $\frac{1}{4}$ & $ \frac{1}{8}$ & $\frac{1}{16}$  & $\frac{1}{32}$ & $\frac{1}{64}$  \\
		\hline
        $1$      & 17 & 16 & 15 & 14 & 12 \\
		$0.1$    & 17 & 16 & 15 & 14 & 13 \\
		$0.01$   & 17 & 16 & 15 & 14 & 14 \\
		$0.001$  & 18 & 16 & 15 & 14 & 14 \\
		$0.0001$ & 17 & 16 & 15 & 14 & 14 \\
		\hline
		\end{tabular}
\begin{tabular}{c c c c c}
        \hline
\multicolumn{5}{ c }{$\widehat{\mathcal{B}_U^E}$}\\ \hline
$\frac{1}{4}$ \hspace{-58pt}\phantom{\backslashbox{$\tau$}{Mesh}}  & $\frac{1}{8}$ & $\frac{1}{16}$  & $\frac{1}{32}$ & $\frac{1}{64}$ \\
		\hline
        22 & 20 & 20 & 20 & 20 \\
		22 & 21 & 20 & 20 & 20 \\
		23 & 21 & 19 & 19 & 19 \\
		23 & 20 & 19 & 19 & 19 \\
		22 & 19 & 19 & 19 & 19 \\
		\hline
		\end{tabular}
\begin{tabular}{l | c c c c c}
		\hline
&\multicolumn{5}{ c }{$\mathcal{B}_L^E$}\\[0.3em] \hline
		 \backslashbox{$\tau$}{$h$} & $\frac{1}{4}$ & $ \frac{1}{8}$ & $\frac{1}{16}$  & $\frac{1}{32}$ & $\frac{1}{64}$  \\
		\hline
        $1$      & 16 & 15 & 15 & 14 & 12 \\
		$0.1$    & 16 & 15 & 15 & 14 & 13 \\
		$0.01$   & 16 & 16 & 15 & 14 & 13 \\
		$0.001$  & 15 & 15 & 15 & 14 & 13 \\
		$0.0001$ & 16 & 16 & 15 & 14 & 13 \\
		\hline
		\end{tabular}
\begin{tabular}{c c c c c}
        \hline
\multicolumn{5}{ c }{$\widehat{\mathcal{B}_L^E}$}\\ \hline
$\frac{1}{4}$ \hspace{-58pt}\phantom{\backslashbox{$\tau$}{Mesh}}  & $\frac{1}{8}$ & $\frac{1}{16}$  & $\frac{1}{32}$ & $\frac{1}{64}$ \\
		\hline
         20 & 21 & 19 & 19 & 16\\
         21 & 20 & 19 & 18 & 17\\
         20 & 20 & 19 & 18 & 17\\
         19 & 20 & 19 & 18 & 17\\
         20 & 20 & 19 & 18 & 18\\
		\hline
		\end{tabular}
\end{center}
\end{table}

Next, we illustrate that the block preconditioners are robust for the cantilever bracket problem as well. Table \ref{tab:prec-elim-2d-nu-K CB Problem} shows iteration counts for the block preconditioners when the values of $\nu$ and $K$ varies for the eliminated system (\ref{eliminated}). The mesh size is fixed as $h=1/64$ and the time step size is $\tau=1$. When $K$ is varying, the number of iterations of the diagonal block preconditioners grows as $K$ approaches $10^{-12}$ but stabilize towards the end. The exact block triangular preconditioners are stable with respect to $K$ and the number of iterations basically stays the same.  But the iteration counts of the inexact block triangular preconditioners increases as $K$ approached to $10^{-12}$.  This is due to the inexactness and, when we solve the diagonal block more accurately, the number iterations will stabilize.  When $\nu$ is varying, the number iterations decreases as $\nu$ gets closer to $0.5$ for all block preconditioners.  However, we want to point out that, when $\nu$ approaches $0.5$, the diagonal block corresponding to the elasticity part becomes more and more difficult to solve and, therefore, the inner preconditioned CG methods might need more iterations since we are using simple UA-AMG method.  Overall, the proposed block preconditioners are robust with respect to the physical parameters which confirms the theoretical results.

Table \ref{tab:prec-elim-2d-h-tau CB Problem} shows the iteration counts for the block preconditioners on the eliminated system for different mesh sizes and time step sizes with $\nu=0.45$ and $K=10^{-7}$. As shown in Table \ref{tab:prec-elim-2d-h-tau CB Problem}, We saw the same behavior as the previous experiments.  The block triangular preconditioners work better than the block diagonal preconditioners. The number of iterations of the inexact block preconditioners is higher than the exact block preconditioners.  Overall, the number of iterations of all the block preconditioners are stable when we varying $h$ and $\tau$, which demonstrates the robustness of the block preconditioners.  In this example, the exact block triangular preconditioners only need $2$-$4$ iterations to achieve tolerance.  This is due to the choice of the physical parameters, the block triangular preconditioned linear system has eigenvalues are clustered near $-1$ and $1$ and, therefore, the Krylov iterative method converges quickly in this case.
 
\begin{table}[htp]
\footnotesize
\begin{center}
\caption{Iteration counts for the block preconditioners for the eliminated system $\mathcal{A}^{E}$ discretized from the cantilever bracket problem with varying physical parameters $K$ and $\nu$.} \label{tab:prec-elim-2d-nu-K CB Problem}
\begin{tabular}{c|c c c c c c}
  \hline
  & \multicolumn{6}{ c }{$\nu = 0.45$ and varying $K$  } \\ \hline
  & $10^{-2}$  & $10^{-4}$ & $10^{-6}$ & $10^{-8}$ & $10^{-10}$ & $10^{-12}$  \\ \hline
	$\mathcal{B}_D^E$           &  4  & 4  & 5  & 14 & 21 & 25  \\
	$\mathcal{B}_U^E$           &  2  & 2  & 2  & 2  & 2  & 2  \\	
	$\mathcal{B}_L^E$           &  3  & 3  & 4  & 4  & 3  & 3  \\
	$\widehat{\mathcal{B}_D^E}$ &  5  & 6  & 10 & 29 & 36 & 38  \\
	$\widehat{\mathcal{B}_U^E}$ &  4  & 4  & 4  & 6  & 9  & 16  \\
	$\widehat{\mathcal{B}_L^E}$ &  5  & 5  & 6  & 8  & 8  & 11  \\
  \hline
  & \multicolumn{6}{ c }{$K = 10^{-7}$ and varying $\nu$ } \\ \hline
                          & $0$ & $0.1$ & $0.2$ & $0.4$  & $0.45$ & $0.49$  \\ \hline
  $\mathcal{B}_D^E$           & 13 & 13 & 13  & 10 & 9  &  6    \\
  $\mathcal{B}_U^E$           &  2 &  2 &  2  &  2 & 2  & 2     \\
  $\mathcal{B}_L^E$           &  5 &  5 &  5  &  4 & 4  & 3     \\
  $\widehat{\mathcal{B}_D^E}$ & 23 & 23 & 23  & 20 & 18 & 11     \\
  $\widehat{\mathcal{B}_U^E}$ &  7 &  7 &  7  &  5 & 5  & 5     \\
  $\widehat{\mathcal{B}_L^E}$ &  9 &  9 &  9  &  7 & 7  & 7     \\
  \hline
\end{tabular}
\end{center}
\end{table}

\begin{table}[ht!]
\footnotesize
\begin{center}
\caption{Iteration counts for the block preconditioners for the eliminated system $\mathcal{A}^{E}$ discretized from the cantilever bracket problem with varying discretization parameters $h$ and $\tau$.}\label{tab:prec-elim-2d-h-tau CB Problem}
\begin{tabular}{l | c c c c c c c c c c}
		\hline
&\multicolumn{5}{ c }{$\mathcal{B}_D^E$}\\[0.3em] \hline
		 \backslashbox{$\tau$}{$h$} & $\frac{1}{4}$ & $ \frac{1}{8}$ & $\frac{1}{16}$  & $\frac{1}{32}$ & $\frac{1}{64}$  \\
		\hline
		$1$       & 15 & 15 & 12 & 10 &  8 \\
        $0.1$     & 22 & 20 & 18 & 16 & 14 \\
		$0.01$    & 25 & 23 & 21 & 20 & 18 \\
		$0.001$   & 25 & 25 & 24 & 22 & 21 \\
		$0.0001$  & 26 & 26 & 26 & 25 & 23 \\
		\hline
		\end{tabular}
\begin{tabular}{c c c c c}
        \hline
\multicolumn{5}{ c }{$\widehat{\mathcal{B}_D^E}$}\\ \hline
$\frac{1}{4}$ \hspace{-58pt}\phantom{\backslashbox{$\tau$}{Mesh}}  & $\frac{1}{8}$ & $\frac{1}{16}$  & $\frac{1}{32}$ & $\frac{1}{64}$ \\
		\hline
		 22 & 23 & 21 & 20 & 17 \\
		 29 & 33 & 36 & 33 & 30 \\
         30 & 33 & 35 & 36 & 37 \\
         31 & 38 & 39 & 34 & 36 \\
         31 & 38 & 39 & 38 & 36 \\
		\hline
		\end{tabular}
\begin{tabular}{l | c c c c c}
		\hline
&\multicolumn{5}{ c }{$\mathcal{B}_U^E $}\\[0.3em] \hline
		 \backslashbox{$\tau$}{$h$} & $\frac{1}{4}$ & $ \frac{1}{8}$ & $\frac{1}{16}$  & $\frac{1}{32}$ & $\frac{1}{64}$  \\
		\hline
       	$1$     & 2 & 2 & 2 & 2 &  2 \\
		$0.1$   & 3 & 2 & 2 & 2 &  2 \\
		$0.01$  & 3 & 2 & 2 & 2 &  2 \\
		$0.001$ & 3 & 2 & 2 & 2 &  2 \\
		$0.0001$& 3 & 2 & 2 & 2 &  2 \\
		\hline
		\end{tabular}
\begin{tabular}{c c c c c}
        \hline
\multicolumn{5}{ c }{$\widehat{\mathcal{B}_U^E}$}\\ \hline
$\frac{1}{4}$ \hspace{-58pt}\phantom{\backslashbox{$\tau$}{Mesh}}  & $\frac{1}{8}$ & $\frac{1}{16}$  & $\frac{1}{32}$ & $\frac{1}{64}$ \\
		\hline
         7 &  6 &  5 &  5 &  5  \\
        10 &  8 &  7 &  6 &  7 \\
        11 & 11 &  9 &  8 &  7 \\
        13 & 12 & 15 & 12 &  9 \\
        14 & 13 & 14 & 15 & 13 \\
		\hline
		\end{tabular}
\begin{tabular}{l | c c c c c}
		\hline
&\multicolumn{5}{ c }{$\mathcal{B}_L^E$}\\[0.3em] \hline
		 \backslashbox{$\tau$}{$h$} & $\frac{1}{4}$ & $ \frac{1}{8}$ & $\frac{1}{16}$  & $\frac{1}{32}$ & $\frac{1}{64}$  \\
		\hline
		$1$      & 4 & 4 & 4 & 4 & 4 \\
		$0.1$    & 4 & 4 & 4 & 4 & 4 \\
		$0.01$   & 4 & 4 & 4 & 3 & 3 \\
		$0.001$  & 4 & 4 & 4 & 3 & 3 \\
		$0.0001$ & 3 & 3 & 3 & 3 & 3 \\
		\hline
		\end{tabular}
\begin{tabular}{c c c c c}
        \hline
\multicolumn{5}{ c }{$\widehat{\mathcal{B}_L^E} $}\\ \hline
$\frac{1}{4}$ \hspace{-58pt}\phantom{\backslashbox{$\tau$}{Mesh}}  & $\frac{1}{8}$ & $\frac{1}{16}$  & $\frac{1}{32}$ & $\frac{1}{64}$ \\
		\hline
         7 &  7 &  7 &  6 &  7 \\
		 9 &  8 &  8 &  7 &  8 \\
		11 & 10 &  9 &  8 &  8 \\
		11 & 12 & 11 & 10 &  8 \\
		10 & 11 & 11 & 10 & 11 \\
		\hline
		\end{tabular}
\end{center}
\end{table}

\section{Conclusion}
In this work, we considered a stabilized hybrid mixed method for the three-field Biot's model. The hybrid mixed finite element method based on the triple
P1-RT0-P0 presented in~\cite{C. Niu} is not uniformly stable with respect to the physical parameters, for instance, when the permeability is small with respect to the mesh
 size, it does not converge. To overcome such problem, we presented a stabilization technique with bubble functions following~\cite{V. Girault,C. Rodrigo}.
  The well-posedness of the stabilized scheme with respect to the energy norm is given independent of the physical and discretization parameters. We then use a spectrally equivalent perturbed bilinear form as suggested in~\cite{C. Rodrigo}. The perturbation of the bilinear form allows for elimination
of the bubble functions; The normal component of Darcy's velocity is discontinuous across the interior edges, so the mass matrix
corresponding to Darcy's velocity is block diagonal. Therefore, the unknowns of the bubble functions and Darcy's velocity can be eliminated by static condensation.
 In fact, the eliminated system is the same size as the P1-RT0-P0 discretization, which is widely used in practice.
 We also prove that the eliminated system is well-posed and efficient preconditioners are developed to solve the resulting linear system. We show theoretically that the block preconditioners are robust with respect to physical and discretization parameters. Finally, numerical experiments validate the accuracy and efficiency of the stabilization method and also demonstrate the robustness of the block preconditioners.

\section*{Appendix. Proof of Theorem \ref{thm:error}}

In this appendix, we provide the proof of Theorem \ref{thm:error}.
\begin{proof}
  First of all, we introduce the notation $ \beta \in L^{2}(\partial \Omega)$ for pressure on the boundary of domain $\Omega$. Here, $L^{2}(\partial \Omega)$ denotes the set of square integrable functions on the boundary of $\Omega$. Then, for the model problem (\ref{e:MODEL}), integrating by parts, we have,
\begin{equation}\label{e:error eqn}
\begin{array}{l}
a(\bm{u},\bm{v})-\alpha(p,\nabla \cdot \bm{v})=(\bm{f},\bm{v}),\ \ \forall \ \bm{v} \in \bm{H}^{1}_{0}(\Omega),\\
\displaystyle{\alpha(\nabla \cdot \partial _{t}\bm{u},q)+\frac{1}{M}(  \partial _{t}p,q) +(\nabla \cdot \bm{w},q) =-(g,q)}, \ \ \forall \ q \in L^{2}(\Omega),\\
 (\bm{w}\cdot \bm{n}_{e},\rho)_{\partial\mathcal{T}_{h}}=0, \ \ \forall \ \rho \in L^{2}(\partial \Omega),\\
-(p,\nabla \cdot \bm{r})-( \beta,\bm{r}\cdot \bm{n}_{e})_{\partial\mathcal{T}_{h}}+(\kappa^{-1}\bm{w},\bm{r})=0, \ \ \forall \ \bm{r} \in H(\text{div},\Omega).
\end{array}
\end{equation}
Choosing $\bm{v} = \bm{v}_{h}$, $q = q_{h}$, $\rho = \rho_{h}$ and $\bm{r} = \bm{r}_{h}$ in (\ref{e:error eqn}) and subtracting these equations from (\ref{e:stableHMFE}) and using the definition of elliptic projections given in (\ref{e:elliptic projections}), we obtain
 \begin{equation}\label{e:error eqn1}
\begin{array}{l}
 a(e_{\bm{u}}^{n},\bm{v}_{h})-\alpha(e_{p}^{n},\nabla \cdot \bm{v}_{h})=0,\\
\displaystyle{\alpha(\nabla \cdot \overline{\partial} _{t} e_{\bm{u}}^{n},q_{h})
+\frac{1}{M}(\overline{\partial}_{t}e_{p}^{n},q_{h})
+(\nabla \cdot e_{\bm{w}}^{n},q_{h})=(R^{n},q_{h})},\\
(e_{\bm{w}}^{n}\cdot \bm{n}_{e},\rho_{h})_{\partial\mathcal{T}_{h}}=0,\\
-(e_{p}^{n},\nabla \cdot\bm{r}_{h})-(e_{\beta}^{n},\bm{r}_{h}\cdot \bm{n}_{e})_{\partial\mathcal{T}_{h}}
+(\kappa^{-1}e_{\bm{w}}^{n},\bm{r}_{h})_{h}
=0.
\end{array}
\end{equation}
where $(R^{n},q_{h})=\displaystyle{\alpha(\nabla \cdot  R_{\bm{u}}^{n},q_{h})+\frac{1}{M}( R_{p}^{n},q_{h})}$, with
\begin{equation*}
\begin{array}{l}
 \displaystyle{R_{\bm{u}}^{n}:=\partial _{t}\bm{u}^{n}-\frac{\overline{\bm{u}}^{n}-\overline{\bm{u}}^{n-1}}{\tau}},\ \
  \displaystyle{R_{p}^{n}:=\partial _{t}p^{n}-\frac{\overline{p}^{n}-\overline{p}^{n-1}}{\tau}}.
\end{array}
\end{equation*}
Then, choosing $\bm{v}_{h}=\overline{\partial} _{t} e_{\bm{u}}^{n}$, $q_{h}=e_{p}^{n}$, $\rho_{h}=e_{\beta}^{n}$, and $\bm{r}_{h}=e_{\bm{w}}^{n}$, respectively, and adding all the equations of (\ref{e:error eqn1}), we have,
 \begin{equation*}
\begin{array}{l}
\displaystyle{\|e_{\bm{u}}^{n}\|_{a}^{2}
+\frac{1}{M}\|e_{p}^{n}\|^{2}+\tau\|e_{\bm{w}}^{n}\|_{h,\kappa^{-1}}^{2}}\\
 \leq \displaystyle{\|e_{\bm{u}}^{n}\|_{a} \|e_{\bm{u}}^{n-1}\|_{a}+\frac{1}{M}\|e_{p}^{n}\|\|e_{p}^{n-1}\|+\tau \| R_{\bm{u}}^{n}\|_{1}\|e_{p}^{n}\|+\frac{\tau}{M}\| R_{p}^{n}\|\|e_{p}^{n}\|}.
\end{array}
\end{equation*}
Due to our choice of finite element spaces and the first equality in (\ref{e:error eqn1}), we have
\begin{equation*}
\begin{array}{l}
\displaystyle{\|e_{p}^{n}\|\leq c\sup_{0\neq\bm{v}_{h} \in \bm{V}_{h}}\frac{(e_{p}^{n},\nabla \cdot \bm{v}_{h} )}{\|\bm{v}_{h} \|_{a}}=
c\sup_{0\neq\bm{v}_{h} \in \bm{V}_{h}}\frac{a(e_{\bm{u}}^{n}, \bm{v}_{h} )}{\|\bm{v}_{h} \|_{a}}=
c\|e_{\bm{u}}^{n}\|_{a}}.
\end{array}
\end{equation*}
Therefore, we have
\begin{equation*}
\begin{array}{l}
\displaystyle{\|e_{\bm{u}}^{n}\|_{a}^{2}
+\frac{1}{M}\|e_{p}^{n}\|^{2}+\tau\|e_{\bm{w}}^{n}\|_{h,\kappa^{-1}}^{2}}\\
 \leq \displaystyle{\|e_{\bm{u}}^{n}\|_{a}\|e_{\bm{u}}^{n-1}\|_{a}+\frac{1}{M}\|e_{p}^{n}\|\|e_{p}^{n-1}\|
 +c\|e_{\bm{u}}^{n}\|_{a}\biggl(\tau \| R_{\bm{u}}^{n}\|_{1}+\frac{\tau}{M}\| R_{p}^{n}\|\biggr)}\\
\leq \displaystyle{\biggl(\|e_{\bm{u}}^{n}\|_{a}^{2}
+\frac{1}{M}\|e_{p}^{n}\|^{2}\biggr)^{1/2}\biggl[\biggl(\|e_{\bm{u}}^{n-1}\|_{a}^{2}
+\frac{1}{M}\|e_{p}^{n-1}\|^{2}\biggr)^{1/2}}\\
\hspace{4mm}+\displaystyle{c\biggl(\tau \| R_{\bm{u}}^{n}\|_{1}+\frac{\tau}{M}\| R_{p}^{n}\|\biggr)\biggr]}.
\end{array}
\end{equation*}
This implies,
\begin{equation*}
\begin{array}{l}
\displaystyle{\biggl(\|e_{\bm{u}}^{n}\|_{a}^{2}
+\frac{1}{M}\|e_{p}^{n}\|^{2}\biggr)^{1/2}}\\
\leq \displaystyle{\biggl(\|e_{\bm{u}}^{n-1}\|_{a}^{2}
+\frac{1}{M}\|e_{p}^{n-1}\|^{2}\biggr)^{1/2}+c\biggl(\tau \|R_{\bm{u}}^{n}\|_{1}+\frac{\tau}{M}\| R_{p}^{n}\|\biggr)},\\
\end{array}
\end{equation*}
and
\begin{equation*}
\begin{array}{l}
\displaystyle{\tau^{1/2}\|e_{\bm{w}}^{n}\|_{h,\kappa^{-1}}\leq \biggl(\|e_{\bm{u}}^{n-1}\|_{a}^{2}
+\frac{1}{M}\|e_{p}^{n-1}\|^{2}\biggr)^{1/2}+c\biggl(\tau \|R_{\bm{u}}^{n}\|_{1}+\frac{\tau}{M}\| R_{p}^{n}\|\biggr)}.\\
\end{array}
\end{equation*}
By recursion, we get,
\begin{equation}\label{e:28}
\begin{array}{l}
\displaystyle{\biggl(\|e_{\bm{u}}^{n}\|_{a}^{2}
+\frac{1}{M}\|e_{p}^{n}\|^{2}\biggr)^{1/2}}\\
\leq \displaystyle{\biggl(\|e_{\bm{u}}^{0}\|_{a}+\frac{1}{M}\|e_{p}^{0}\|\biggr)^{1/2}+c\biggl(\tau \sum_{n=1}^{N} \|R_{\bm{u}}^{n}\|_{1}+\frac{\tau}{M}\sum_{n=1}^{N}\| R_{p}^{n}\|\biggr)},
\end{array}
\end{equation}
and
\begin{equation}\label{e:288}
\begin{array}{l}
\displaystyle{\tau^{1/2}\|e_{\bm{w}}^{n}\|_{h,\kappa^{-1}}}\\
\leq \displaystyle{\biggl(\|e_{\bm{u}}^{0}\|_{a}^{2}
+\frac{1}{M}\|e_{p}^{0}\|^{2}\biggr)^{1/2}+c\biggl(\tau \sum_{n=1}^{N} \|R_{\bm{u}}^{n}\|_{1}+\frac{\tau}{M}\sum_{n=1}^{N}\| R_{p}^{n}\|\biggr)}.\\
\end{array}
\end{equation}
Combining (\ref{e:28}) and (\ref{e:288}), we have the estimate
\begin{equation}\label{e:eupw}
\begin{array}{l}
\|(e_{\bm{u}}^{n},e_{p}^{n},e_{\bm{w}}^{n})\|_{\tau,h}\\
 \leq \displaystyle{c\biggl(\|e_{\bm{u}}^{0}\|_{a^{D}}+\frac{1}{M}\|e_{p}^{0}\|+\tau \sum_{n=1}^{N} \| R_{\bm{u}}^{n}\|_{1}+\frac{\tau}{M}\sum_{n=1}^{N}\| R_{p}^{n}\|\biggr)}.
\end{array}
\end{equation}
Now, following the same procedures of Lemma 8 in~\cite{C. Rodrigo2}, we have
\begin{equation}\label{Ru}
\begin{array}{l}
\displaystyle{\sum_{n=1} ^{N}\|R_{\bm{u}}^{n}\|_{1}\leq c\biggl(\int_{0}^{t_{N}}\|\partial_{tt}\bm{u}\|_{1}dt+\frac{1}{\tau}\int_{0}^{t_{N}}\|\partial_{t}\rho_{\bm{u}}\|_{1}dt\biggr)},
\end{array}
\end{equation}
 and similarly,
\begin{equation}\label{Rp}
\begin{array}{l}
\displaystyle{\sum_{n=1} ^{N}\|R_{p}^{n}\|\leq c\biggl(\int_{0}^{t_{N}}\|\partial_{tt}p\|dt+\frac{1}{\tau}\int_{0}^{t_{N}}\|\partial_{t}\rho_{p}\|dt\biggr)}.
\end{array}
\end{equation}
Then, the error estimate (\ref{e:error}) follows from (\ref{e:eupw}), (\ref{Ru}), (\ref{Rp}), (\ref{e:projection property}) and triangle inequality.
\end{proof}




\end{document}